\documentclass[a4paper,12pt]{article}

\usepackage{amssymb,amsthm,amsmath,amscd,amsfonts,bbm,mathrsfs}

\usepackage[all]{xy}
\SelectTips{cm}{}

\newtheorem{satz}{Theorem}
\theoremstyle{plain}
\newtheorem{thm}{Theorem}[section]
\newtheorem{df}[thm]{Def\/inition}
\newtheorem{prop}[thm]{Proposition}
\newtheorem{cor}[thm]{Corollary}
\newtheorem{lemma}[thm]{Lemma}

\theoremstyle{remark}
\newtheorem{remark}[thm]{Remark}

\newcommand{\eqn}{\begin{equation}}
\newcommand{\eeqn}{\end{equation}}

\newcommand{\BBB}{\mathcal{B}}
\newcommand{\DDD}{\mathcal{D}}
\newcommand{\FFF}{\mathcal{F}}
\newcommand{\GGG}{\mathcal{G}}
\newcommand{\HHH}{\mathcal{H}}
\newcommand{\JJJ}{\mathcal{J}}

\newcommand{\MMM}{\mathcal{M}}
\newcommand{\PPP}{\mathcal{P}}
\newcommand{\TTT}{\mathcal{T}}
\newcommand{\WWW}{\mathcal{W}}

\newcommand{\FF}{\mathbb{F}}
\newcommand{\ZZ}{\mathbb{Z}}

\newcommand{\Fa}{\mathfrak{a}}
\newcommand{\Fp}{\mathfrak{p}}
\newcommand{\Fm}{\mathfrak{m}}

\DeclareMathOperator{\Coker}{Coker} 
\DeclareMathOperator{\Cris}{Cris} 
\DeclareMathOperator{\Def}{Def} 
\DeclareMathOperator{\GL}{GL} 
 
\DeclareMathOperator{\Hom}{Hom} 
\DeclareMathOperator{\Spec}{Spec}
\DeclareMathOperator{\Ker}{Ker}  
\DeclareMathOperator{\Image}{Im}  
\DeclareMathOperator{\Lie}{Lie}  
\DeclareMathOperator{\Frob}{Frob}  
  
\DeclareMathOperator{\id}{id}  
  
\DeclareMathOperator{\Nil}{Nil}  
\DeclareMathOperator{\Res}{Res}

\DeclareMathOperator{\Aut}{Aut}

\DeclareMathOperator{\uAut}{\underline{Aut}}
\DeclareMathOperator{\uEnd}{\underline{End}}
\DeclareMathOperator{\uHom}{\underline{Hom}}
\DeclareMathOperator{\uIsom}{\underline{Isom}}

\begin{document}

\title{Truncated Barsotti-Tate Groups and Displays}
\author{Eike Lau, Thomas Zink}
\maketitle

\begin{abstract}
We define truncated displays over rings in which a prime $p$
is nilpotent, we associate crystals to truncated displays, and
we define functors from truncated displays to
truncated Barsotti-Tate groups.
\end{abstract}

\section*{Introduction}

The aim of this paper is to analyse the relation between truncated
displays and truncated Barsotti-Tate groups.

Let us recall the notion of a truncated display as used in [L2]. 
We fix a prime number $p$. 
Let $R$ be a commutative ring with unit such that $pR = 0$. We
denote by $W(R)$ the ring of $p$-typical Witt vectors and by $W_n(R)$  
the truncated ones. We write $^{F}\xi$ and $^{V}\xi$ for the Frobenius
and Verschiebung of an element $\xi \in W(R)$. 

A display $\mathcal{P}$ over $R$ may be given by
the following data: Two locally free finitely generated $W(R)$-modules
$T$ and $L$ and a Frobenuis linear isomorphism 
\begin{displaymath}
\Phi: T \oplus L \rightarrow T \oplus L.
\end{displaymath}
In the introduction we will assume that $T \cong W(R)^{d}$ and $L
\cong W(R)^{c}$ are free $W(R)$-modules. Then $\Phi$ 
is given by an invertible block matrix  
\begin{displaymath}
\begin{array}{cc}
  \left(
\begin{array}{cc}
A & B\\
C & D
\end{array}
\right) \in \GL_{d+c}(W(R)).  
\end{array}
\end{displaymath}
\begin{displaymath}
\Phi 
(\left(
\begin{array}{c}
t\\
l 
\end{array}
\right)) =
\left(
\begin{array}{cc}
A & B\\
C & D
\end{array}
\right)
\left(
\begin{array}{c}
^{F}t\\
^{F}l
\end{array}
\right),
\end{displaymath}
where $t \in W(R)^{d}$ and $l \in W(R)^{c}$. The height of the display
is $h = d + c$. 

Assume a second display $\mathcal{P}'$ is given by a block
matrix. Then a morphism $\mathcal{P} \rightarrow \mathcal{P}'$ is the
same as a block matrix:
\begin{displaymath}
\left(
\begin{array}{cc}
X & \mathfrak{J}\\
Z & Y
\end{array}
\right) \in M(h' \times h, W(R)))
\end{displaymath}
of size $h'\times h$ such that the following relation holds: 
\begin{equation}\label{AbI1e} 
\left(
\begin{array}{cc}
A' & B'\\
C' & D'
\end{array}
\right)
\left(
\begin{array}{cc}
^F\!X &  \;\mathfrak{J}\\
p\,^F\! Z & ^F Y
\end{array}
\right)
=
\left(
\begin{array}{cc}
X & ^V \mathfrak{J}\\
Z & Y
\end{array}
\right)
\left(
\begin{array}{cc}
A & B\\
C & D
\end{array}
\right)
\end{equation} 
This is the description of the category of displays in terms of
matrices. 

To define truncated displays of level $n$ we take all
matrices with coefficients in $W_n(R)$.  Since $pR = 0$ there is a
Frobenius $F: W_n(R) \rightarrow W_n(R)$. Therefore the definition of
a morphism (\ref{AbI1e}) makes perfectly sense if
we take $V$ to be the composition $W_n(R) \overset{V}{\rightarrow}
W_{n+1}(R) \overset{Res}{\longrightarrow} W_n(R)$. 

Let $\mathcal{P}$ be a truncated display of level $n$ and let $m \leq n$. 
The restriction morphism $W_n(R) \rightarrow W_m(R)$ 
applied to the matrix of $\mathcal{P}$ gives a truncated display $\mathcal{P}(m)$ 
of level $m$, called the truncation of $\mathcal{P}$. This is a functor.

\subsubsection*{A partial inverse of the display functor}

By \cite{Lau-Smoothness} we have a functor 
\begin{displaymath}
   \Phi_{n,R}: \BBB\TTT_n(R) \rightarrow \mathcal{D}_n(R)
\end{displaymath} 
from the category of truncated $p$-divisible  groups of level $n$ over $R$
to the category of truncated displays of level $n$ over $R$. 
For $m\le n$ we denote the truncation of $G\in\BBB\TTT_n(R)$ by
$G(m)=G[p^m]$, the kernel of $p^m\id_G$.
The functors $\Phi_{n,R}$ are compatible with the truncation functors.

We say that $G \in \BBB\TTT_n(R)$ is nilpotent of order $e\ge 0$ if the iterate of the Frobenius 
$F_G^{e+1}:G\to G^{(p^{e+1})}$ is zero on the first truncation $G(1)$, 
or equivalently if $F_G^e$ induces zero on $\Lie(G^\vee)$, where $G^\vee$ is the Cartier dual of $G$. 
This condition can be formulated in terms of the truncated display of $G$.
By restriction we obtain a functor  
\begin{displaymath}
   \Phi_{n,R}: \BBB\TTT^{(e)}_n(R) \rightarrow \mathcal{D}^{(e)}_n(R)
\end{displaymath}
from the category of truncated $p$-divisible groups which are nilpotent of order $e$ 
to the category of truncated displays which are nilpotent of order $e$.

\begin{satz}\label{AbI1t} 
Assume that $pR = 0$. Let $n,m,e \geq 0$ be natural numbers such that
$n \geq m(e+1)$. 
There is a functor 
\begin{displaymath}
   BT_{m,R} :  \mathcal{D}^{(e)}_n(R) \rightarrow \BBB\TTT^{(e)}_m(R)
\end{displaymath}
such that we have natural isomorphisms
\begin{align*}
BT_{m,R}(\Phi_{n,R}(G)) \cong G(m) & \qquad\text{for}\quad  G \in \BBB\TTT^{(e)}_n(R),\\
\Phi_{m,R}(BT_{m,R}(\mathcal{P})) \cong \mathcal{P}(m) & \qquad\text{for}\quad  \mathcal{P}
\in \mathcal{D}^{(e)}_n(R).
\end{align*}
\end{satz}
This is proved in section 3 (Lemma \ref{Le:BTm} and Proposition \ref{Pr-BT-Phi}).
The construction of the functor $BT_{m,R}$ is a variant of the functor $BT_R$
from nilpotent displays to $p$-divisible groups in \cite{Zink-DFG}.

\subsubsection*{Extended versions of truncated displays}

We define truncated displays also for rings $R$ in which $p$ is nilpotent, 
and we develop a deformation theory for truncated displays,
based on a notion of relative truncated displays for a divided power extension $S\to R$.

Let $\Cris_m(R)$ be the category of  all $pd$-thickenings $S \rightarrow R$ with kernel 
$\mathfrak{a}$ such that $p^{m}\mathfrak{a} = 0$.
The central result about deformations is Propositon \ref{Pr-lift}, which implies
 that all liftings of a truncated display $\PPP$ over $R$ to a
relative truncated display $\tilde\PPP$ for $S\to R$ are isomorphic; morover if $\PPP$ is nilpotent
of order $e$, then the truncation of\/ $\tilde\PPP$ by $m(e+1)+1$ steps is unique up to \emph{unique}
isomorphism. This can be viewed as a refined
version of \cite[Theorem 44]{Zink-DFG} about deformations of nilpotent displays.  

This leads to the construction of a crystal associated to a nilpotent truncated display: 
Let $\mathcal{P}$ be a truncated display of level
$n$ over $R$ which is nilpotent of order $e$ and assume that
$n > m(e+1)+1$.
We define a crystal $\mathbb{D}_{\mathcal{P}}$ 
of locally free $\mathcal{O}$-modules on $\Cris_m(R)$ by the rule
$\mathbb D_{\mathcal P}(S)=\tilde P(1)$; see \eqref{AbCr1e}.

\begin{satz}\label{Theorem2}
Let $t \geq n$ such that
$p^tW_{n}(R) = 0$. Then there is a functor 
\begin{displaymath}
   \Phi_n: \BBB\TTT_t(R) \rightarrow \mathcal{D}_n(R).
\end{displaymath}
Let $X$ be a $p$-divisible group $R$ such that $X_{R/pR}$ is nilpotent of order $e$.
Let  $\mathcal{P} = \Phi_n(X(t))$. 
For an object $S \rightarrow R$ of $\Cris_m(R)$ where 
$n > m(e+1)+1$, we have a canonical isomorphism
\begin{displaymath}\label{AbLF2e}
   \mathbb{D}_{X}(S) \cong \mathbb{D}_{\mathcal{P}}(S)
\end{displaymath} 
where $\mathbb D_X$ is the Grothendieck-Messing crystal
of $X$ defined in {\rm \cite{Messing}}.
\end{satz}

This is proved in Propositions \ref{Pr:DXDP} and \ref{AbLF1p}.
Using Theorem \ref{Theorem2} we prove a version
of Theorem \ref{AbI1t} for algebras over $\mathbb{Z}/p^{m}\mathbb{Z}$;
see Proposition \ref{Pr:BT-Phi2}.

All functors $BT$ and $\Phi$ are exact with respect to the obvious exact
structures on the categories.

\subsubsection*{Additional results and remarks}

1) Let $R$ be a ring with $pR=0$.
For truncated $p$-divisible groups $G_1,G_2$ of level $n$ over $R$
with associated truncated displays $\PPP_i=\Phi_{n,R}(G_i)$ we consider 
the group scheme of vanishing homomorphisms
\[
\uHom^o(G_1,G_2)=\Ker[\uHom(G_1,G_2)\to\Hom(\PPP_1,\PPP_2)],
\]
and similarly for automorphisms. 
Following \cite{Lau-Smoothness},
elementary arguments show that if $G$ is nilpotent of order $e$ and if
$n\ge m(e+1)$, then the truncation homomorphism
\[
\uAut^o(G)\to\uAut (G(m))
\]
is trivial. One deduces 
that there are functors $BT_{m,R}$ as in
Theorem \ref{AbI1t}; see Proposition \ref{Pr:Gamma}.
This proof does not make the functors $BT_{m,R}$ explict, 
but in exchange it avoids the question
of showing directly that $\Phi_{n,R}$ and $BT_{m,R}$ are compatible.

\medskip
2) We obtain a new proof of the equivalence between formal $p$-divisible
groups over $R$ and nilpotent displays over $R$, proved first in \cite{Zink-DFG}
when $R$ is excellent and in \cite{Lau-Display} in general:
The limit over $m$ of Theorem \ref{AbI1t} gives the case $pR=0$, 
and the general case follows easily by deformation theory.

\medskip
3) As a by-product of the proof of Proposition 2.3
we also obtain a purely local proof of the smoothness of the functor $\Phi_n$, 
viewed as a morphism of algebraic stacks over $\FF_p$; 
see Proposition \ref{Pr-lift-disp} and page \pageref{Se:Smoothness}.

\medskip

In an appendix we prove that truncated displays satisfy fpqc descent.

\tableofcontents

\section{The category of truncated displays}

We fix a prime number $p$. Let $R$ be a ring such that $p$ is
nilpotent in $R$.  For fixed $n \in \mathbb{N}$ let $W_n(R)$ be the
ring of truncated Witt vectors.  We consider the ring homomorphism
induced by the 
restriction $\Res: W_{n+1}(R) \rightarrow W_n(R)$ and the
Frobenius $F: W_{n+1}(R) \rightarrow W_n(R)$:
\begin{equation}\label{Ab1e} 
(\Res , F): W_{n+1} (R) \rightarrow W_n(R) \times W_n(R). 
\end{equation} 
The image of this ring homomorphism will be denoted by
$\mathcal{W}_n(R)$. 
The kernel consists of the elements $~^{V^n} [s]$, where $s \in R$ and $ps
= 0$. It follows easily that $R \mapsto \mathcal{W}_n(R)$ is a sheaf for the 
f.p.q.c.-topology; see the Appendix.

The two projections will be denoted by 
\begin{displaymath} 
\Res : \mathcal{W}_n(R) \rightarrow W_n(R), \qquad F: \mathcal{W}_n(R)
\rightarrow W_n(R). 
\end{displaymath}  
If $pR = 0$ then $\Res : \mathcal{W}_n(R) \rightarrow W_n(R)$ is an isomorphism.

Let $I_{n+1} = ~^VW_n(R) \subset W_{n+1} (R)$. This is a 
$\mathcal{W}_n(R)$-module by  
\begin{displaymath} 
\xi ~^V \eta = ~^V( ~^F \xi \eta), \qquad \text{for }\; \xi \in
\mathcal{W}_n(R), \; \eta \in W_n(R). 
\end{displaymath} 
The inverse of the Verschiebung defines a bijective map
$V^{-1}: I_{n+1} \rightarrow W_n(R)$, which is
$F$-linear with respect to $F: \mathcal{W}_n(R)
\rightarrow W_n(R)$. We denote by $\kappa:  I_{n+1} \rightarrow
\mathcal{W}_n (R)$ the map induced by \eqref{Ab1e}. The cokernel of
$\kappa$ is $\mathbf{w}_0 \circ \Res: \mathcal{W}_n(R) \rightarrow
W_n(R) \rightarrow  R$.  

\begin{df}\label{Def-tr-disp}
A truncated display $\mathcal{P}$ of level $n$ over a ring $R$ in which $p$
is nilpotent consists of 
$(P,Q,\iota, \epsilon, F, \dot{F})$. Here $P$ and $Q$ are 
$\mathcal{W}_n (R) $-modules, 
\begin{displaymath} 
\iota: Q \rightarrow P, \qquad \epsilon: I_{n+1} \otimes_{\mathcal{W}_n
  (R)} P \rightarrow Q,
\end{displaymath} 
are $\mathcal{W}_n (R)$-linear maps, and 
\begin{displaymath} 
F: P \rightarrow  W_n(R) \otimes_{\mathcal{W}_n (R)} P, \qquad \dot{F}:
Q \rightarrow W_n(R) \otimes_{\mathcal{W}_n (R)} P
\end{displaymath} 
are $F$-linear maps. (The tensor products are taken with
respect to $\Res$.) 

The following conditions are required:

\begin{enumerate}
\item[(i)] $P$ is a finitely generated projective $\mathcal{W}_n
  (R)$-module. 
\item[(ii)] The maps $\iota \circ \epsilon$ and $\epsilon \circ
  ({\id_{I_{n+1}}} \otimes \iota) $ are the multiplication maps (via
  $\kappa$).
\item[(iii)] The cokernels of $\iota$ and $\epsilon$ are finitely
  generated projective $R$-modules.
\item[(iv)] There is a commutative diagram
\begin{equation}\label{AbDef1e} 
\xymatrix@M+0.2em{
I_{n+1} \otimes_{\mathcal{W}_n (R)} P
\ar[rr]^-{\epsilon}  
\ar[dd]_{\tilde{F}} &  & Q, \ar[lldd]^-{\dot{F}}\\
 &  &\\
 W_n(R) \otimes_{\mathcal{W}_n (R)} P    & &
}
\end{equation}
where $\tilde{F}$ is defined by $\tilde{F}(~^V\eta \otimes x) = \eta
F(x).$
\item[(v)] $\dot{F}(Q) $ generates $W_n(R) \otimes_{\mathcal{W}_n (R)}
  P$ as a $W_n(R)$-module. 
\item[(vi)] We have an exact sequence
\begin{displaymath} 
0 \rightarrow Q/\Image \epsilon \overset{\iota}\longrightarrow
P/\kappa(I_{n+1})P \rightarrow P/\iota(Q) \rightarrow 0.
\end{displaymath} 
(Only the injectivity of the second arrow is a requirement.)
\end{enumerate}
Truncated displays of level $n$ over $R$ form an additive category in
an obvious way, which we denote by $\mathcal{D}_n(R)$.  
\medskip

The surjective $R$-linear map $P/\kappa(I_{n+1})P\to P/\iota(Q)$
is called the Hodge filtration of the truncated display $\mathcal P$. 
\end{df}

When $pR = 0$ we have just $F$-linear maps $F: P \rightarrow P$, 
$\dot{F}: Q \rightarrow P$ as in the case of displays, but $Q$ is not a 
submodule of $P$. 

It follows from the above axioms that
\begin{displaymath}
   F(\iota(y)) = p \dot{F}(y), \quad \text{for}\; y \in Q
\end{displaymath}
using $\eta=1$ and $x=\iota(y)$ in (iv).

\begin{df}\label{AbND1d}
Let $\PPP=(P,Q,\iota,\epsilon,F,\dot F)$ be a truncated diaplay of level $n$ over $R$.
A normal decomposition for $\PPP$ consists of $(T,L,u,v)$
where $T$ and $L$ are finitely generated projective $\mathcal{W}_n(R)$-modules 
with isomorphisms 
\begin{displaymath} 
u:P \cong T \oplus L, \qquad 
v:Q \cong I_{n+1}  \otimes_{\mathcal{W}_n (R)} T
\oplus L,
\end{displaymath} 
such that the maps
\begin{displaymath}
\begin{array}{ccc} 
\iota:  I_{n+1}  \otimes_{\mathcal{W}_n (R)} T \oplus L & \rightarrow & T
\oplus L\\[2mm]
\epsilon:  I_{n+1}  \otimes_{\mathcal{W}_n (R)} T \oplus I_{n+1}
\otimes_{\mathcal{W}_n (R)} L & \rightarrow  & I_{n+1}
\otimes_{\mathcal{W}_n (R)} T \oplus L
\end{array} 
\end{displaymath}
are given as follows: $\iota$ is the multiplicationon on the first summand
and the identity on the second summand, while $\epsilon$
is the identity on the first summand and the
multiplication on the second summand.   
\end{df}

Unlike in the case of displays, the isomorphism $u:T\oplus L\cong P$
does not in general determine $v$.

\begin{prop}\label{Pr-norm-dec}
Every truncated display has a normal decomposition.
\end{prop}

\begin{proof}
Let  $\mathcal{P} = (P,Q,\iota, \epsilon, F, \dot{F})$ be a  truncated
display of level $n$ over $R$.

We take a projective $\mathcal{W}_n(R)$-module $T$ which lifts
$P/\iota (Q)$ and a projective $\mathcal{W}_n(R)$-module $L$ which
lifts $Q/\Image \epsilon$.  We choose liftings $T \rightarrow P$ and
$L \rightarrow Q$ of the natural projections $P \rightarrow
P/\iota(Q)$ and $Q \rightarrow Q/\Image \epsilon$. 
We consider the natural homomorphism 
\begin{displaymath} 
T \oplus L \rightarrow P,
\end{displaymath} 
induced by $\iota$ on the second summand.  
By (vi) this becomes an isomorphism if we tensor it by $R
\otimes_{\mathcal{W}_n(R)}$, and therefore it is an isomorphism. 

Now we consider the natural homomorphism 
\begin{equation}\label{ND2e} 
\nu: I_{n+1} \otimes_{\mathcal{W}_n(R)} T \oplus L \rightarrow Q,
 \end{equation} 
induced  by $\epsilon$ on the first factor.

We note that $\epsilon: I_{n+1} \otimes L \rightarrow Q$ is the multiplication by (ii). This shows that the image of $\nu$ contains the image of $\epsilon$.   Therefore the homomorphism \eqref{ND2e} is surjective.   We will show it is injective.  

Let us denote by $\Phi: T \rightarrow W_n(R) \otimes_{\mathcal{W}_n(R)} P$ the restriction of $F$ to $T$ and by $\dot{\Phi}: L \rightarrow W_n(R) \otimes_{\mathcal{W}_n(R)} P$ the composite of $\dot{F}$ with $L \rightarrow Q$.

  We denote by $\tilde{\Phi}: I_{n+1} \otimes_{\mathcal{W}_n(R)} T \rightarrow W_n(R) \otimes_{\mathcal{W}_n(R)} P$ the map defined  by $\tilde{\Phi} (~^V\xi \otimes t) = \xi \otimes \Phi(t)$. Then we obtain a commutative diagram 

\begin{displaymath}
   \xymatrix{
I_{n+1} \otimes T \oplus L \ar[rr]^{\nu} \ar[rd]_{\tilde{\Phi} \oplus
  \dot{\Phi}} & & Q \ar[ld]_{\dot{F}}\\ 
 & W_n(R) \otimes_{\mathcal{W}_n(R)} P. & \\
}
\end{displaymath}

\bigskip 
\noindent
We now assume without loss of generality that $L$ and $T$ are free; 
see Lemma \ref{Le-W-Zar} below. 
Let $t_1, \ldots t_d$ be a basis of $T$ and $l_1, \ldots, l_c$ be a basis 
of $L$. Since by the diagram $\tilde{\Phi} \oplus \dot{\Phi}$ is an $F$-linear
epimorphism we conclude that $\Phi(t_1), \ldots, \Phi(t_d),
\dot{\Phi}(l_1), \ldots, \dot{\Phi}(l_c)$ is a basis of $W_n(R)
\otimes_{\mathcal{W}_n(R)} P$.  

Consider an element in the kernel of $\nu$:
\begin{equation}\label{ND3e} 
\sum ~^V \xi_i \otimes t_i + \sum \eta_j l_j \in I_{n+1} \otimes T
\oplus L, \quad \xi_i \in W_n(R), \; \eta \in \mathcal{W}_n(R).
\end{equation} 
Since $\tilde\Phi \oplus \dot{\Phi}$ applied to this element must be zero in
$W_n(R) \otimes_{\mathcal{W}_n(R)} P$ we conclude that $\xi_i = 0$.  

On the other hand the restriction to $\nu$ to $0 \oplus L$ is
injective because $\iota \circ \nu$ is the injection $0 \oplus L
\subset P$. This proves that the element \eqref{ND3e} is zero.
\end{proof}

\begin{lemma}
\label{Le-W-Zar}
Let $S \subset R$ be a multiplicatively closed system. We denote by
$[S] \subset \mathcal{W}_n(R)$ the multipicatively closed system which consists
of the Teichm\"uller representatives $[s] \in \mathcal{W}_n(R)$ of
elements $s \in S$. 
 
Let $R'=R[S^{-1}]$. Then $\mathcal W_n(R')=\mathcal W_n(R)[[S]^{-1}]$.
\end{lemma}

\begin{proof}
The corresponding fact for $W_{n+1}$ is known.
The kernel of $W_{n+1}(R)\to\mathcal W_n(R)$ is the module
of $p$-torsion elements $R[p]$, considered as an $R/pR$-module
via $\Frob^n$. The lemma follows.
\end{proof}

\begin{remark}
Proposition \ref{Pr-norm-dec} implies
that Definition \ref{Def-tr-disp}
coincides with the definition of truncated displays in
\cite[Definition 3.4]{Lau-Smoothness}  if $pR=0$. 
Indeed, the conditions on
$(P,Q,\iota,\epsilon)$ imposed here are weaker than 
those of \cite{Lau-Smoothness}, which are equivalent
to the existence of a normal decomposition.
But the difference disappears in the presence of $(F,\dot F)$.
See also Lemma \ref{Le-std-proj} below.
\end{remark}

Any normal decomposition is obtained as follows: Choose liftings
$T' \rightarrow P/\iota(Q)$ and $L' \rightarrow Q/\Image
\epsilon $, to projective $\mathcal{W}_n(R)$-modules and extend them
to homomorphisms $T' \rightarrow P$ and $L' \rightarrow Q$. Then
$\iota: L' \rightarrow P$ is injective and $P = T' \oplus L'$ is a
normal decomposition.  

We note that the maps $F$ and $\dot{F}$ are uniquely determined by
their linearisations:
\begin{displaymath} 
\begin{array}{lccc} 
F^{\sharp} : & W_n(R) \otimes_{F, \mathcal{W}_n (R)} P & \rightarrow &
W_n(R) \otimes_{\mathcal{W}_n (R)} P \\
\dot{F}^{\sharp}:  & W_n(R) \otimes_{F, \mathcal{W}_n (R)}  Q &
\rightarrow & W_n(R) \otimes_{\mathcal{W}_n (R)} P
\end{array} 
\end{displaymath} 
 
As in the case of displays, from the normal decomposition we obtain an
isomorphism of $W_n(R)$-modules
\begin{equation}\label{Ab4e}
F^{\sharp} \oplus \dot{F}^{\sharp} : (W_n(R) \otimes_{F, \mathcal{W}_n
  (R)} T) \oplus (W_n(R) \otimes_{F, \mathcal{W}_n (R)} L) \rightarrow
W_n(R) \otimes_{\mathcal{W}_n (R)} P.
\end{equation} 

We write the last map as a matrix 
\begin{equation}\label{AbStr-e}
   \left(
\begin{array}{cc}
A & B\\
C & D
\end{array}
\right),
\end{equation}
where 

\begin{displaymath}   
\begin{array}{ccc}
A: W_n \otimes_{F,\mathcal{W}_n} T  \rightarrow  W_n
\otimes_{\mathcal{W}_n} T, &&
B : W_n \otimes_{F,\mathcal{W}_n} L
 \rightarrow  W_n \otimes_{\mathcal{W}_n} T,\\ 
C : W_n \otimes_{F,\mathcal{W}_n} T   \rightarrow  W_n
\otimes_{\mathcal{W}_n} L, &&
D : W_n \otimes_{F,\mathcal{W}_n} L  \rightarrow  W_n
\otimes_{\mathcal{W}_n} L, 
\end{array}
\end{displaymath}
are $W_n(R)$-linear maps. 

Conversely, by the following construction, 
a matrix \eqref{AbStr-e} which is an isomorphism of
$W_n(R)$-modules defines a truncated display of level $n$. 

We set $\dot{\sigma} = V^{-1}: I_{n+1} \rightarrow W_n(R)$ 
and consider this as an isomorphism von $\mathcal{W}_n(R)$-modules: 
\begin{displaymath}
   I_{n+1} \rightarrow W_n(R)_{[F]},
\end{displaymath}
where the last index denotes restriction of scalars by $F$. 
For an arbitrary $\mathcal{W}_n(R)$-module, $\dot{\sigma} $ induces an
isomorphism denoted by the same letter
\begin{displaymath}
   \dot{\sigma}: I_{n+1} \otimes_{\mathcal{W}_n(R)} T \rightarrow
   W_n(R) \otimes_{F,\mathcal{W}_n(R)} T.
\end{displaymath}
This is $F$-linear with respect to $F: \mathcal{W}_n(R)
\rightarrow W_n(R)$. 

We also use the notation $\sigma$ for the
$F$-linear maps
\begin{displaymath}
\begin{array}{rcc}
\sigma: T & \rightarrow & W_n \otimes_{F,\mathcal{W}_n(R)} T\\
\ell & \mapsto & 1 \otimes \ell.
\end{array}
\end{displaymath}

To obtain a truncated display of level $n$ from a matrix
\eqref{AbStr-e} we set 
\begin{displaymath}
   P = T \oplus L, \quad Q = I_{n+1} \otimes T \oplus L.
\end{displaymath}
Then we have obvious maps $\iota$ and $\epsilon$ as in Proposition
\ref{Pr-norm-dec}. For vectors
\begin{displaymath}
\left(
\begin{array}{c}
t\\
\ell
\end{array}
\right) \in 
 T \oplus L, \quad  \left(
\begin{array}{c}
y\\
\ell
\end{array}
\right) \in  I_{n+1} \otimes T \oplus L
\end{displaymath} 
We define $F$ and $\dot{F}$ as follows
\begin{equation}\label{AbSt1e}
   F \left(
\begin{array}{c}
t\\
\ell
\end{array}
\right) =
   \left(
\begin{array}{cc}
A & pB\\
C & pD
\end{array}
\right)
\left(
\begin{array}{c}
\sigma(t)\\
\sigma(\ell)
\end{array}
\right)
\end{equation}

\begin{equation}\label{AbSt2e}
 \dot{F} \left(
\begin{array}{c}
y\\
\ell
\end{array}
\right) =
   \left(
\begin{array}{cc}
A & B\\
C & D
\end{array}
\right)
\left(
\begin{array}{c}
\dot{\sigma}(y)\\
\sigma(\ell)
\end{array}
\right)
\end{equation}

\begin{df}
Let\/ $T$ and $L$ be finitely generated projective $\mathcal{W}_n(R)$-modules.
Assume we are given a matrix \eqref{AbStr-e} of homomorphisms $A,B,C,D$ 
which is invertible. If we define $F$, $\dot{F}$ by 
the formulas \eqref{AbSt1e} 
and \eqref{AbSt2e} we obtain a truncated display 
$\mathcal{S}(T,L;A,B,C,D)$ level $n$ which we call a standard truncated 
display of level $n$ over the ring $R$.  
\end{df}

We will now describe a homomorphism of standard truncated displays:
\begin{displaymath}
\alpha: \mathcal{S}(T,L;A,B,C,D) \rightarrow \mathcal{S}(T',L';A',B',C',D'). 
\end{displaymath}
We write $P = T \oplus L$ and so on. The morphism $\alpha$ is given by
two module homomorphism $\alpha_0: P \rightarrow P'$ and 
$\alpha_1: Q \rightarrow Q'$ which have to be compatible with the maps 
$\iota, \epsilon$ and $F, \dot{F}$. From the compatability with the first two 
maps we  see that there are homomorphisms
\begin{equation}\label{AbSt1e1}
\begin{array}{ll}
X \in \Hom_{\mathcal{W}_n(R)} (T, T'), & U \in 
\Hom_{\mathcal{W}_n(R)} (L, I_{n+1} \otimes_{\mathcal{W}_n(R)} T'),\\
Z \in \Hom_{\mathcal{W}_n(R)} (T, L'), & Y \in \Hom_{\mathcal{W}_n(R)} (L, L'),
\end{array}
\end{equation}  
such that the homomorphisms $\alpha_0$ and $\alpha_1$ are given by the formulas:
\begin{equation}\label{AbSt3e}
\alpha_1 
\left(
\begin{array}{c}
\xi \otimes t\\
     \ell
\end{array}
\right) = \left(
\begin{array}{c}
\xi \otimes Xt + U\ell\\
\kappa(\xi)Zt + Y\ell
\end{array}
\right) \in I_{n+1} \otimes_{\mathcal{W}_n(R)} T' \oplus L',
\end{equation}

\begin{equation}\label{AbSt4e}
\alpha_0  \left( 
\begin{array}{c}
t\\
\ell
\end{array}
\right) = \left( 
\begin{array}{cc}
X & \hat{U}\\
Z & Y
\end{array}
\right) 
\left( 
\begin{array}{c}
t\\
\ell
\end{array}
\right),
\end{equation}
for $t \in T$, $\ell \in L$, $\xi \in I_{n+1}$.
Here $\hat U$ is the composition of $U$ with the multiplication
$I_{n+1}\otimes_{\mathcal W_n(R)}T'\to T'$.
We consider the map 
\begin{displaymath}
\dot{\sigma} \oplus \sigma : (I_{n+1} \otimes_{\mathcal{W}_n(R)} T') \oplus L' 
\rightarrow W_n(R) \otimes_{F, \mathcal{W}_n(R)} T' \oplus 
W_n(R) \otimes_{F, \mathcal{W}_n(R)} L'.
\end{displaymath}
If we apply this to the vector \eqref{AbSt3e} we obtain:
\begin{displaymath}
\left( 
\begin{array}{cc}
\sigma(X) & \dot{\sigma}(U)\\
p\sigma(Z) & \sigma(Y)
\end{array}
\right) \left( 
\begin{array}{c}
\dot{\sigma}(\xi)\sigma(t)\\
\sigma  (\ell) 
\end{array}
\right) \in W_n(R) \otimes_{F, \mathcal{W}_n(R)} T' \oplus 
W_n(R) \otimes_{F, \mathcal{W}_n(R)} L'.
\end{displaymath} 
Here we use the notation 
\begin{displaymath}
\sigma (X) = \id_{W_n(R)} \otimes_{\mathcal{W}_n(R)} X: 
W_n(R)\otimes_{F, \mathcal{W}_n(R)} T \rightarrow W_n(R)\otimes_{F, \mathcal{W}_n(R)} T',
\end{displaymath}
and similarly $\sigma(Z)$ and $\sigma(Y)$. The composition
\begin{displaymath}
L \overset{U}{\rightarrow} I_{n+1} \otimes_{\mathcal{W}_n(R)} T'
\overset{\dot{\sigma}}{\rightarrow} W_n(R) \otimes_{F, \mathcal{W}_n(R)} T'
\end{displaymath}
is linear with respect to $F: \mathcal{W}_n(R) \rightarrow W_{n}(R)$. 
Its linearisation is 
\begin{equation}
\label{Eq-sigmadotU}
\dot{\sigma}(U): W_n(R) \otimes_{F, \mathcal{W}_n(R)} L \rightarrow 
W_n(R) \otimes_{F, \mathcal{W}_n(R)} T'.
\end{equation}
The pair $\alpha_0, \alpha_1$ is a morphism of truncated displays iff
the following diagram is commutative:
\begin{equation}\label{AbSt5e}
\begin{CD}
Q @>{\dot{F}}>> W_n(R) \otimes_{\mathcal{W}_n(R)} P\\
@V{\alpha_1}VV @VV{\id \otimes \alpha_0}V\\
Q' @>{\dot{F}'}>> W_n(R) \otimes_{\mathcal{W}_n(R)} P'
\end{CD}
\end{equation}
We have just computed
\begin{displaymath}
\dot{F}'\circ \alpha_1 \left( 
\begin{array}{c}
\xi \otimes t\\
\ell
\end{array}
\right) = 
\left( 
\begin{array}{cc}
A' & B' \\
C' & D' 
\end{array}
\right) 
\left( 
\begin{array}{cc}
\sigma(X) & \dot{\sigma}(U)\\
p\sigma(Z) & \sigma(Y)
\end{array}
\right) \left( 
\begin{array}{c}
\dot{\sigma}(\xi)\sigma(t)\\
\sigma  (\ell) 
\end{array}
\right)
\end{displaymath}
If we tensor \eqref{AbSt4e} with $W_n(R) \otimes_{\mathcal{W}_n(R)}$ we obtain 
the matrix
\begin{displaymath}
\left( 
\begin{array}{cc}
\bar{X} & \bar{U}\\
\bar{Z} & \bar{Y}  
\end{array}
\right). 
\end{displaymath}
Then the commutativity of \eqref{AbSt5e} is equivalent with the equation
\begin{equation}\label{AbSt7e}
\left( 
\begin{array}{cc}
A' & B' \\
C' & D' 
\end{array}
\right) 
\left( 
\begin{array}{cc}
\sigma(X) & \dot{\sigma}(U)\\
p\sigma(Z) & \sigma(Y)
\end{array}
\right) =
\left( 
\begin{array}{cc}
\bar{X} & \bar{U}\\
\bar{Z} & \bar{Y}  
\end{array}
\right)
\left( 
\begin{array}{cc}
A & B \\
C & D 
\end{array}
\right).
\end{equation}

Let us summarise the preceding considerations.

\begin{df}
We define the category $s\,\mathcal{D}_n(R)$ 
of standard truncated displays of level $n$ over $R$ as follows.
Its objects are data $(T,L;A,B,C,D)$ as above, and a morphism 
\begin{displaymath}
(T,L;A,B,C,D) \rightarrow (T',L';A',B',C',D').   
\end{displaymath}
is a matrix of homomorphisms \eqref{AbSt1e1} which satisfies the equation 
\eqref{AbSt7e}. 
\end{df}

\begin{prop}\label{Pr-std-rep}
There is a functor 
\begin{displaymath}
\mathcal S:s\,\mathcal{D}_n(R) \rightarrow \mathcal{D}_n(R),
\end{displaymath}
and this functor is an equivalence of categories. \qed
\end{prop}

Remark: We have not defined explicitly the composition in the category $s\,\mathcal{D}_n(R)$,
but the composition is uniquely determined by the requirement that $\mathcal S$ is a functor.
See also Definition \ref{Def:MnR} below.

\subsection{Base change and truncation functors}

Let $\phi: R \rightarrow S$ be a homomorphism of rings in which $p$ is nilpotent. Let 
$\mathcal{S}(T,L;A,B,C,D)$ be a  standard truncated display of level $n$. 
We set $\tilde{T} = \mathcal{W}_n(S) \otimes_{\mathcal{W}_n(R)} T$ and  
$\tilde{L} = \mathcal{W}_n(S) \otimes_{\mathcal{W}_n(R)} L$. By tensoring the maps
\eqref{AbStr-e} by $W_n(S)\otimes_{W_n(R)}$ we obtain an object of 
$s\,\mathcal{D}_n(S)$. It is easily checked that this gives a functor
\begin{displaymath}
\beta_s: s\,\mathcal{D}_n(R) \rightarrow s\,\mathcal{D}_n(S),
\end{displaymath}
which is the base change functor for standard truncated displays. Therefore
by Proposition \ref{Pr-std-rep} we also get a base change functor 
\begin{equation}\label{AbBC1e}
\beta: \mathcal{D}_n(R) \rightarrow \mathcal{D}_n(S).  
\end{equation}
To make this canonical one can proceed in the standard way. Let $\mathcal{P} 
\in \mathcal{D}_n(R)$. Then we consider the category $\mathcal{C}$ whose 
objects are isomorphisms $\mathcal{S} \rightarrow \mathcal{P}$, where 
$\mathcal{S} \in s\,\mathcal{D}_n(R)$, and we define $\beta(\mathcal{P})$
as the projective limit over $\mathcal{C}$:
\begin{displaymath}
\beta(\mathcal{P}) = \underset{\leftarrow }{\lim}\, \beta_s(\mathcal{S}). 
\end{displaymath} 
If $\mathcal{P} = (P,Q,\iota, \epsilon, F, \dot{F})$ we write 
$\beta(\mathcal{P}) = (P_S, Q_S, \iota_S, \epsilon_S, F_S, \dot{F}_S)$. 
We note that there is a canonical isomorphism $P_S = 
\mathcal{W}_n(S) \otimes_{\mathcal{W}_n(R)} P$.  

In the same way we can define truncation functors
\begin{equation}\label{Abtrunc1e}
\tau_m: \mathcal{D}_{m+1}(R) \rightarrow \mathcal{D}_m(R).
\end{equation}
They are compatible with the base change functors. We will write
$\mathcal{P}(m) := \tau_m(\mathcal{P})$. More generally, if 
$\mathcal{P}$ is a truncated display of level $n \geq m$ we will
consider the truncation $\mathcal{P}(m)$.  

One could also define the truncation and base change functors by a universal
property without referring to standard representations (but the proof that
the functors exist uses Proposition \ref{Pr-std-rep}).
Namely, for $\phi:R\to S$ as above and for truncated displays $\PPP$ over $R$
and $\PPP'$ over $S$ of level $n$ one defines homomorphisms $\PPP\to\PPP'$
over $\phi$ in the obvious way. Then we have a universal homomorphism 
$\PPP\to\beta(\PPP)$ over $\phi$. A similar remark applies to the truncation
functors.

\subsection{Matrix description}

For simplicity we will often assume that the modules $T$ resp.\ $L$ are free
of rank $d$ and $c$; see Lemma \ref{Le-W-Zar}. 
We fix isomorphisms $T \cong \mathcal{W}(R)^d$
and $L \cong \mathcal{W}(R)^c$. Then a standard truncated display with
normal decomposition given by $T$ and $L$ is 
determined by the invertible matrix 
\begin{displaymath} 
\left(
\begin{array}{cc}
A & B\\
C & D
\end{array}
\right) 
\in \GL_{d+c} (W_n(R))
\end{displaymath} 
which defines the map \eqref{Ab4e}. 
For $~^V \underline{\eta} \in I_{n+1}^d$, $\underline{\zeta} \in
\mathcal{W}_n(R)^d$, and $\underline{\xi} \in
\mathcal{W}_n(R)^c$, we consider the vectors: 
\begin{displaymath} 
\left(
\begin{array} {r}
^V \underline{\eta}\\
\underline{\xi}\\
\end{array} 
\right)
\in I_{n+1} \otimes_{\mathcal{W}_n} T \oplus L = Q,
\qquad
\left(
\begin{array} {r}
\underline{\zeta}\\
\underline{\xi}\\
\end{array} 
\right)
\in T\oplus L=P.
\end{displaymath} 
Then $F$ and $\dot{F}$ can be written in matrix form

\begin{displaymath} 
F (
\left(
\begin{array} {r}
\underline{\zeta}\\
\underline{\xi}\\
\end{array} 
\right)
) = 
\left(
\begin{array}{cc}
A & pB\\
C & pD
\end{array}
\right) 
\left(
\begin{array} {r}
^F\underline{\zeta}\\
^F \underline{\xi}\\
\end{array} 
\right).
\end{displaymath}

\begin{displaymath} 
\dot{F} (
\left(
\begin{array} {r}
^V \underline{\eta}\\
\underline{\xi}\\
\end{array} 
\right)
) = 
\left(
\begin{array}{cc}
A & B\\
C & D
\end{array}
\right) 
\left(
\begin{array} {r}
\underline{\eta}\\
^F \underline{\xi}\\
\end{array} 
\right).
\end{displaymath}

As for displays we represent morphisms by matrices.
Let $\mathcal{P}'$ be a second truncated display with a given normal
decomposition 
\begin{displaymath} 
P' = T' \oplus L', \quad Q' = I_{n+1}  \otimes_{\mathcal{W}_n (R)} T'
\oplus L'.
\end{displaymath}
We fix isomorphisms $T' \cong \mathcal{W}_n(R)^{d'}$  and $L' \cong
\mathcal{W}_n(R)^{c'}$. We assume that $\mathcal{P}'$ is defined by the
matrix
\begin{displaymath} 
\left(
\begin{array}{cc}
A' & B'\\
C' & D'
\end{array}
\right) 
\in \GL_{d'+c'} (W_n(R)).
\end{displaymath} 

A morphism $\alpha:\mathcal P\to\mathcal P'$ is given by a matrix
\begin{equation}\label{Ab3e} 
\left(
\begin{array}{cc}
X & ^V \mathfrak{J}\\
Z & Y
\end{array}
\right),
\end{equation}  
 where $X \in \Hom_{\mathcal{W}_n (R)} (T,T')$, $Y \in
 \Hom_{\mathcal{W}_n (R)}(L,L')$, $Z \in \Hom_{\mathcal{W}_n (R)} (T, L')$
 $^V \mathfrak{J} \in \Hom_{\mathcal{W}_n (R)} (L, I_{n+1} \otimes T')$. The matrices
 $X$, $Y$, $Z$ have coefficients in $\mathcal{W}_n(R)$ and $\mathfrak{J}$ has
 coefficients in $W_n(R)$.

The maps $Q \rightarrow Q'$ and $P \rightarrow P'$ induced by $\alpha$ 
are given by the matrices
\begin{displaymath}
\left(
\begin{array}{cc}
X & ~^V \mathfrak{J}\\
Z & Y
\end{array}
\right) :
I_{n+1}  \otimes_{\mathcal{W}_n (R)} T \oplus L \rightarrow I_{n+1}
\otimes_{\mathcal{W}_n (R)} T' \oplus L'
\end{displaymath}
and
\begin{displaymath}
\left(
\begin{array}{cc}
X & \kappa(^V \mathfrak{J})\\
Z & Y
\end{array}
\right) :
T \oplus L \rightarrow  T' \oplus L',
\end{displaymath}
where the first matrix needs a little interpretation. 

The matrix \eqref{Ab3e} defines a morphism of truncated displays
iff the following equation holds:
\begin{equation}\label{Ab2e} 
\left(
\begin{array}{cc}
A' & B'\\
C' & D'
\end{array}
\right) 
\left(
\begin{array}{rr}
^F\!X &  \mathfrak{J}\\
p\,^F\! Z & ^F Y
\end{array}
\right)
=
\left(
\begin{array}{cc}
\Res(X) & ^V \bar{\mathfrak{J}}\\
\Res(Z) & \Res(Y)
\end{array}
\right)
\left(
\begin{array}{cc}
A & B\\
C & D
\end{array}
\right) 
\end{equation} 
Here $\bar{\mathfrak{J}}$ is the restriction of $\mathfrak{J}$ to a matrix with coefficients
in $W_{n-1} (R)$.

This equation shows in particular that $\mathfrak{J}$ is already uniquely
determined by $X,Y,Z$ and $\kappa(^V \mathfrak{J})$. Therefore a morphism of
truncated displays is already uniquely determined by the induced
$\mathcal{W}_n$-module homomorphism $P \rightarrow P'$, 
i.e.\ we have proved the following.

\begin{lemma}
\label{Le-faithful}
For two truncated displays $\PPP$ and $\PPP'$ of level $n$ over $R$
the forgetful homomorphism
\begin{equation} 
\Hom_{\mathcal{D}_n} (\mathcal{P},\mathcal{P}') \to
\Hom_{\mathcal{W}_n(R)}(P,P').
\end{equation} 
is injective.
\qed
\end{lemma}

\begin{df}
\label{Def:MnR}
Let $\mathcal{M}_n(R)$ be the category whose objects are
invertible block matrices:
\begin{displaymath} 
\mathfrak{F} = \left(
\begin{array}{cc}
A & B\\
C & D
\end{array}
\right) \in GL_h(W_n(R)), 
\end{displaymath} 
where $A$ and $D$ are square matrices of arbitrary size and 
whose morphisms $\mathfrak{F} \rightarrow
\mathfrak{F}'$ are block matrices $\mathfrak{X}$ of the form
\eqref{Ab3e} which satisfy the relation \eqref{Ab2e}. If
$\mathfrak{X}': \mathfrak{F}' \rightarrow \mathfrak{F}''$ is a second
morphism then the composite $\mathfrak{X}'\circ \mathfrak{X}$ is the
matrix
\begin{displaymath} 
\left(
\begin{array}{cc}
X' & ^V \mathfrak{J}'\\
Z' & Y'
\end{array}
\right) \circ 
\left(
\begin{array}{cc}
X & ^V \mathfrak{J}\\
Z & Y
\end{array}
\right) :=
\left(
\begin{array}{cc}
X'X + \kappa(^V\mathfrak{J}')Z & ^V(^F\!X' \mathfrak{J} + \mathfrak{J}'^{\,F}Y)\\
Z'X + Y'Z  & Z'\kappa(^V\mathfrak{J}) + Y'Y
\end{array}
\right).
\end{displaymath}
\end{df}

We have a fully faithful functor
$\mathcal{M}_n(R) \rightarrow s\,\mathcal{D}_n(R)
\xrightarrow\sim \mathcal{D}_n(R)$.
The essential image consists of the truncated
displays such that the modules in the exact sequence (vi) of
Definition \ref{Def-tr-disp} are free $R$-modules.  

\subsection{Nilpotent truncated displays}

Let $\mathcal{P}$ be a truncated display of level $n$ over a ring $R$. 
There is a unique homomorphism of $W_n(R)$-modules
\begin{equation}\label{AbVis1e}
V^{\sharp}: W_n(R)\otimes_{\mathcal{W}_n(R)} P \rightarrow  
W_n(R)\otimes_{F, \mathcal{W}_n(R)} P, 
\end{equation}
such that for each $y \in Q$
\begin{displaymath}
V^{\sharp}(\dot{F}(y)) = 1 \otimes \iota (y).
\end{displaymath}
The existence follows as in the case of displays by using a normal 
decomposition. One deduces from the last equation that
\begin{displaymath}
V^{\sharp}(\xi F(x)) = p\xi \otimes x.
\end{displaymath}
We assume now that $pR = 0$. Then we have $W_n(R) = \mathcal{W}_n(R)$
and $I_n = \kappa (I_{n+1})$. The homomorphism \eqref{AbVis1e} takes
the form 
\begin{displaymath}
V^{\sharp}: P \rightarrow W_n(R) \otimes_{F,W_n(R)} P.
\end{displaymath}
Iterating the last morphism we obtain for each natural number $N$:
\begin{equation}\label{AbVis2e}
(V^{N})^{\sharp}: P \rightarrow W_n(R) \otimes_{F^{N},W_n(R)} P.
\end{equation}
In the case $pR=0$ we say that the truncated display $\mathcal P$ is
nilpotent if for large $N$ the image  
of the map \eqref{AbVis2e} is zero modulo $I_n$. Equivalently we can say
that the map induced by \eqref{AbVis2e}
\begin{displaymath}
(V^{N})^{\sharp}: P/I_nP \rightarrow R \otimes_{\Frob^{N}, R} P/I_nP 
\end{displaymath}
is zero.
In the case of general $R$ we call $\mathcal P$ nilpotent if its base
change to $R/pR$ is nilpotent.
Assume that $pR=0$.
By definition the image of $V^{\sharp}$ coincides with the image of
the homomophism $W_n(R) \otimes_{F,W_n(R)} Q \rightarrow 
W_n(R) \otimes_{F,W_n(R)} P$ induced by $\iota$. This implies that the map 
$V^{\sharp}: P \rightarrow R \otimes_{\Frob, R}P/\iota(Q)$ is
zero. Therefore $V^{\sharp}$ 
induces a homomorphism $P/I_nP \rightarrow R \otimes_{\Frob, R}
\iota(Q)/I_nP$. 
By restriction of we obtain the homomorphism
\begin{equation}\label{AbVis3e} 
\bar{V}^{\sharp}: \iota(Q)/I_nP \rightarrow R \otimes_{\Frob, R}
\iota(Q)/I_nP. 
\end{equation}

\begin{df}
Let $\mathcal{P}$ be a nilpotent truncated display of level $n$ over a 
ring $R$. If $pR = 0$ the nilpotence order of $\mathcal{P}$ is 
the smallest natural number $e \geq 0$ such that 
\begin{displaymath}
(\bar{V}^{e})^{\sharp} = 0,
\end{displaymath}
for the iterate of \eqref{AbVis3e}.
If $R$ is arbitrary the order of nilpotence of $\mathcal{P}$ is the order
of nilpotence of the base change $\mathcal{P}_{R/pR}$.

The same makes sense for displays.
\end{df}

\begin{lemma}
\label{Le-nilpotent}
Let $\mathcal{P}$ be a truncated display of level $n$ over $R$, which
is given by the block matrix 
\begin{displaymath} 
\left(
\begin{array}{cc}
A & B\\
C & D
\end{array}
\right)
\end{displaymath} 
We denote the inverse matrix by 
\begin{displaymath} 
\left(
\begin{array}{cc}
\breve{A} & \breve{B}\\
\breve{C} & \breve{D}
\end{array}
\right).
\end{displaymath} 
Let $\hat{D}_0$ be the image of $\mathbf{w}_0(\breve{D})$ in
$R/pR$. 

Then $\mathcal{P}$ has order of nilpotence $\le e$ iff
\begin{displaymath} 
\hat{D}_0^{(p^{e-1})}  \cdot \ldots \cdot \hat{D}_0^{(p)} \hat{D}_0 = 0,
\end{displaymath} 
where the upper index $(p^i)$ means that we take the $p^i$-power of
all entries.
\end{lemma}

\begin{proof}
This is similar to the case of nilpotent displays in \cite[(15)]{Zink-DFG}.
\end{proof}

We note that the order of nilpotence does not change under truncation
(\ref{Abtrunc1e}). Also it does not change by base change 
with a ring homomorphism
$R\to S$ for which $R/pR\to S/pS$ is injective.

In the case of $p$-divisible groups, the nilpotence order can be expressed
as follows. Recall that a $p$-divisible group $G$ over a ring $R$ with $pR=0$ 
is a formal group iff the Frobenius $F$ is nilpotent on $G(1)$.

\begin{df}
\label{Def:nilpotence-pdiv}
Let $G$ be a formal $p$-divisible group over a ring $R$ in which $p$ is nilpotent.
If $pR=0$ the nilpotence order of $G$ is the smallest natural number $e\ge 0$ such
that $F^{e+1}$ is zero on $G(1)$. If $R$ is arbitrary the order of nilpotence of $G$
is the order of nilpotence of $G_{R/pR}$.

The same applies to infinitesimal truncated $p$-divisible groups.
\end{df}

\begin{lemma}
\label{Le:nilpotence}
Let $G$ be a formal $p$-divisible group over a ring $R$ with $pR=0$.
Then the nilpotence order of $G$ is equal to the nilpotence order of
the associated nilpotent display $\PPP$.
\end{lemma}

\begin{proof}
Let $G^\vee$ be the dual of $G$.
Then $\Lie(G^\vee)$ is isomorphic to $Q/I_RP$ such that the Verschiebung 
$V$ of $G^\vee$ corresponds to $\bar V^\sharp:Q/I_RP\to (Q/I_RP)^{(p)}$.
This is the only relation between $G$ and $\PPP$ we need.
We define $\bar G$ by the exact sequence 
$0\to G[F]\to G(1)\to\bar G\to 0$ of finite locally free group schemes. 
Then $\bar G^\vee\cong G^\vee[F]$,
in particular $\Lie(G^\vee)=\Lie(\bar G^\vee)$. 

Now $F^{e+1}$ is zero on $G(1)$ iff $F^e$ is zero on $\bar G$ iff 
$V^e:(\bar G^\vee)^{(p^e)}\to\bar G^\vee$ is zero. By the equivalence
between affine group schemes of finite presentation annihilated by $F$
and $p$-Lie algebras, this homomorphism $V^e$ is zero iff it induces zero on the Lie algebra,
which holds iff $(\bar V^e)^\sharp$ is zero.
\end{proof}

\section{Relative truncated displays}

We consider now a surjective ring homomorphism $S \rightarrow R$, such
that $p$ is 
nilpotent in $S$ and the kernel $\mathfrak{a}$ is endowed with divided
powers. We will say that $S/R$ is a $pd$-thickening. We set 
\begin{displaymath} 
\mathcal{J}_{n+1}  = W_{n+1} (\mathfrak{a}) + I_{n+1}(S) \subset W_{n+1}(S).
\end{displaymath} 
Let $\kappa: \mathcal{J}_{n+1} \rightarrow \mathcal{W}_n(S)$ be the
 homomorphism induced by \eqref{Ab1e}. 

The divided powers define an isomorphism of $W_{n+1} (S)$-modules
\begin{displaymath} 
W_{n+1} (\mathfrak{a}) = \prod_{i=0}^{n} \mathfrak{a}_{[\mathbf{w}_i]}
\end{displaymath} 
which is given by the divided Witt polynomials $W_{n+1}(\mathfrak{a}) 
\rightarrow \mathfrak{a}$. 
The first factor on the right hand side will be also written as
$\tilde{\mathfrak{a}} \subset W_{n+1} (\mathfrak{a})$. Since
$\tilde{\mathfrak{a}}$ is an $S$-module it is a fortiori a
  $\mathcal{W}_n(S)$-module and therefore $\mathcal{J}_{n+1} =
  \tilde{\mathfrak{a}} \oplus I_{n+1} (S)$ is a  $\mathcal{W}_n
  (S)$-module too.  

The map
$V^{-1}: I_{n+1}(S) \rightarrow W_{n}(S)$ extends uniquely to:
\begin{equation}\label{Ab4e1} 
\dot{\sigma}: \mathcal{J}_{n+1} \rightarrow W_n(S), \quad
\text{where}\; \dot{\sigma}(\tilde{\mathfrak{a}}) = 0.
\end{equation} 
We will write $\sigma$ for the Frobenius map $F: \mathcal{W}_n(S)
\rightarrow W_n(S)$. 
We can define relative truncated displays of level $n$ with respect to $S
\rightarrow R$ as before:

\begin{df}\label{Ab2d}
A relative truncated display $\mathcal{P}$ of level $n$ for $S\to R$ 
consists of $(P,Q,\iota, \epsilon, F, \dot{F})$.
Here $P$ and $Q$ are $\mathcal{W}_n(S)$-modules,
\begin{displaymath}
   \iota: Q \rightarrow P, \quad \epsilon:
   \mathcal{J}_{n+1}\otimes_{\mathcal{W}_n(S)} P \rightarrow Q
\end{displaymath}  
are $\mathcal{W}_n(S)$-linear maps,  and
\begin{displaymath}
\begin{array}{rccc}  
F: & P & \rightarrow  & W_n(S) \otimes_{\mathcal{W}_n(S)} P\\
\dot{F}: & Q & \rightarrow & W_n(S) \otimes_{\mathcal{W}_n(S)} P
\end{array} 
\end{displaymath} 
are $\sigma$-linear maps. As in Definition \ref{Def-tr-disp}
we define a map
\begin{displaymath}
\begin{array}{lccc}
\tilde{F}: & \mathcal{J}_{n+1} \otimes_{\mathcal{W}_n(S)} P 
& \rightarrow & W_n(S) \otimes_{\mathcal{W}_n(S)} P \\
 & \tau \otimes x & \mapsto & \dot{\sigma}(\tau) \otimes Fx.
\end{array}
\end{displaymath}
We require that the following properties hold:
\begin{enumerate}
\item[(i)] The $\mathcal{W}_n(S)$-module $P$ is projective and finitely 
generated. 

\item[(ii)] The compositions $\iota\circ\epsilon$ and $\epsilon\circ(\id\otimes\iota)$
are the multiplication maps.

\item[(iii)] The cokernels of $\iota$ and $\epsilon$ are finitely generated projective
$R$-modules.

\item[(iv)] The diagram similar to \eqref{AbDef1e} is commutative,
i.e.\ we have $\dot F\circ\epsilon=\tilde F$.

\item[(v)] The image $\dot{F}(Q) $ generates $W_n(S)
  \otimes_{\mathcal{W}_n (S)} P$  as a  $W_n(S)$-module. 

\item[(vi)] The following sequence is exact:
\begin{equation}\label{AbHodge2e} 
0 \rightarrow Q/\Image \epsilon \overset{\iota}\longrightarrow
P/\kappa(\mathcal{J}_{n+1})P \rightarrow P/\iota(Q) \rightarrow 0.
\end{equation}
\end{enumerate}
Relative truncated displays of level $n$ for $S\to R$ form an additive category
in an obvious way, which we denote by $\mathcal D_n(S/R)$.
\medskip

The surjective $R$-linear map 
$P/\kappa(\mathcal{J}_{n+1})P \rightarrow P/\iota(Q)$
is called the Hodge filtration of the relative truncated display $\PPP$.
\end{df}

As before one can show that there is a normal decomposition 
\begin{displaymath} 
P = T \oplus L, \qquad  Q = \mathcal{J}_{n+1}
\otimes_{\mathcal{W}_n(S)} T \oplus L,
\end{displaymath} 
where $T$ and $L$ are $\mathcal{W}_{n}(S)$-modules. The map 
\begin{displaymath}
F^{\sharp} \oplus \dot{F}^{\sharp} : (W_n(S) \otimes_{F, \mathcal{W}_n
  (S)} T) \oplus (W_n(S) \otimes_{F, \mathcal{W}_n (S)} L) \rightarrow
W_n(S) \otimes_{\mathcal{W}_n (S)} P
\end{displaymath}
is an isomorphism of $W_{n}(S)$-modules which we write in matrix 
form as before; see \eqref{AbStr-e}:
\begin{displaymath}
   \left(
\begin{array}{cc}
A & B\\
C & D
\end{array}
\right)
\end{displaymath} 
This leads to the notion of a standard relative truncated display of level 
$n$ with respect to $S/R$,  
$\mathcal{S}(T,L;A,B,C,D)$. The data are the same as for standard truncated  
displays over the ring $S$, but the notion 
of a morphism changes. A morphism of standard relative truncated displays of
level $n$
\begin{displaymath}
\mathcal{S}(T,L;A,B,C,D) \rightarrow \mathcal{S}(T',L';A',B',C',D')
\end{displaymath}
is given by four homomorphisms of $\mathcal{W}_{n}(S)$-modules:
\begin{displaymath}
\begin{array}{ll}
X \in \Hom_{\mathcal{W}_n(S)} (T, T'), & U \in 
\Hom_{\mathcal{W}_n(R)} (L, \mathcal{J}_{n+1} \otimes_{\mathcal{W}_n(S)} T'),\\
Z \in \Hom_{\mathcal{W}_n(S)} (T, L'), & Y \in \Hom_{\mathcal{W}_n(S)} (L, L'),
\end{array}
\end{displaymath}
which satisfy the following relation:
\begin{displaymath}
\left( 
\begin{array}{cc}
A' & B' \\
C' & D' 
\end{array}
\right) 
\left( 
\begin{array}{cc}
\sigma(X) & \dot{\sigma}(U)\\
p\sigma(Z) & \sigma(Y)
\end{array}
\right) =
\left( 
\begin{array}{cc}
\bar{X} & \bar{U}\\
\bar{Z} & \bar{Y}  
\end{array}
\right)
\left( 
\begin{array}{cc}
A & B \\
C & D 
\end{array}
\right).
\end{displaymath}
Here in the relative case $\dot{\sigma}$ is induced by the morphism 
\eqref{Ab4e1} as in \eqref{Eq-sigmadotU}, and $\bar{U}$ is induced by $\kappa: \mathcal{J}_{n+1} 
\rightarrow \mathcal{W}_{n}(S)$ as follows:
\begin{displaymath}
W_n(S) \otimes_{\mathcal{W}_{n}(S)} L  \overset{\id \otimes U}{\longrightarrow}
W_n(S) \otimes_{\mathcal{W}_{n}(S)} \mathcal{J}_{n+1} \otimes_{\mathcal{W}_{n}(S)} T' 
\rightarrow W_n(S) \otimes_{\mathcal{W}_{n}(S)} T'.
\end{displaymath} 
As in the case of truncated displays (Proposition \ref{Pr-std-rep}) we see that
the category of standard relative truncated displays is equivalent to the 
category of relative truncated displays. Using this, again we define truncation functors: 
\begin{displaymath}
\mathcal{D}_{n+1}(S/R) \rightarrow \mathcal{D}_{n}(S/R).
\end{displaymath}
We also have obvious reduction functors which are compatible with truncation
\begin{displaymath}
\mathcal{D}_n(S) \rightarrow \mathcal{D}_n(S/R) \rightarrow \mathcal{D}_n(R).
\end{displaymath}
For a morphism of $pd$-thickenings
\begin{displaymath}
\begin{CD}
S @>>> S'\\
@VVV  @VVV\\
R @>>> R'
\end{CD}
\end{displaymath}
we have a base change functor $\mathcal{D}_n(S/R) \rightarrow \mathcal{D}_n(S'/R')$.

\subsection{Matrix description}

Assume that $T$ and $L$ are free $\mathcal W_n(S)$-modules.
If we fix isomorphisms $T \cong \mathcal{W}_n(S)^d$ and $L \cong
\mathcal{W}_n(S)^c$, the relative truncated display is given by a
matrix in $GL_{d+c}(W_n(S))$ as before:
\begin{equation}\label{Ab5e} 
 \dot{F} (
\left(
\begin{array} {r}
\underline{\tau}\\
\underline{\ell}\\
\end{array} 
\right)
) = 
\left(
\begin{array}{cc}
A & B\\
C & D
\end{array}
\right) 
\left(
\begin{array} {r}
\dot{\sigma} (\underline{\tau})\\
\sigma(\underline{\ell})\\
\end{array} 
\right),
\quad \underline{\tau} \in \mathcal{J}_{n+1}^d, \; \underline{\ell}
\in \mathcal{W}_n(S)^c.
\end{equation}   
Let $\mathcal{P}' = (P',Q',\iota', \epsilon', F', \dot{F}')$ be a
second relative truncated display. Consider a 
normal decomposition $P' = T' \oplus L'$ with $T' \cong \mathcal{W}_n
(S)^{d'}$ and $L' \cong \mathcal{W}_n (S)^{c'}$. Let  
\begin{displaymath} 
\left(
\begin{array}{cc}
A' & B'\\
C' & D'
\end{array}
\right) \in GL_{d'+c'}(W_n(S))
 \end{displaymath} 
be the matrix of $\mathcal{P'}$. A homomorphism $\alpha: \mathcal{P}
\rightarrow \mathcal{P'}$ is a matrix
\begin{equation}\label{Ab5e1} 
\left(
\begin{array}{cc}
X & J\\
Z & Y
\end{array}
\right),
\end{equation} 
 where $X \in \Hom_{\mathcal{W}_n (S)} (T,T')$, $Y \in
 \Hom_{\mathcal{W}_n (S)}(L,L')$, $Z \in \Hom_{\mathcal{W}_n (S)} (T,
 L')$, 
 $ J \in \Hom_{\mathcal{W}_n (S)} (L, \mathcal{J}_{n+1} \otimes T')$, i.e.\ the matrices
 $X$, $Y$, $Z$ have coefficients in $\mathcal{W}_n(S)$ and $J$ has
 coefficients in $\mathcal{J}_{n+1}$, which satisfies the relation
\begin{equation}\label{Ab6e} 
\left(
\begin{array}{cc}
A' & B'\\
C' & D'
\end{array}
\right) 
\left(
\begin{array}{cc}
\sigma(X) & \dot{\sigma} (J)\\
p\sigma(Z)  & \sigma(Y)
\end{array}
\right)
= 
\left(
\begin{array}{cc}
\bar{X} & \bar{J}\\
\bar{Z} & \bar{Y}
\end{array}
\right)
\left(
\begin{array}{cc}
A & B\\
C & D
\end{array}
\right)
\end{equation}  
Here the bar denotes the image under the restriction map
$\mathcal{W}_n (S) \rightarrow W_n(S)$ resp. $\mathcal{J}_{n+1}
\rightarrow \mathcal{W}_n (S) \rightarrow W_n(S)$. 

A morphism $\alpha$ given by a matrix \eqref{Ab5e1} induces a
homomorphism $P \rightarrow P'$ which is given by the matrix 
\begin{displaymath} 
\left(
\begin{array}{cc}
X & \kappa( J)\\
Z & Y
\end{array}
\right).
\end{displaymath} 
This matrix already determines the matrix \eqref{Ab5e1} uniquely. Indeed,
from $\kappa(J)$ one obtains $\bar{J}$. Then the equation \eqref{Ab6e} 
determines $\dot{\sigma}(J)$. Since the intersection of the kernels of
the two maps
$\mathcal{J}_{n+1} \rightarrow W_n(S)$ given by $\dot{\sigma}$ and
$J \mapsto \bar{J}$ is zero, this determines $J$. 
   
As for truncated displays we conclude:

\begin{lemma}
\label{Le-faithful-rel}
For relative truncated displays $\PPP$ and $\PPP'$ of level $n$ for
$S\to R$ the forgetful homomorphism
\begin{equation}\label{Ab6e1} 
\Hom_{\mathcal{D}_n(S/R)} (\mathcal{P}, \mathcal{P}') \to
\Hom_{\mathcal{W}_n(S)} (P, P')
\end{equation} 
is injective.
\end{lemma}

Let us denote by $\mathcal{D}_n(S/R)$ the category of relative
truncated displays with respect to $S \rightarrow R$ and by
$\mathcal{M}_n(S/R)$ the corresponding category of matrices. 
The objects in  $\mathcal{M}_n(S/R)$ are just invertible block
matrices in $GL_h(W_n(S))$. The compositions of morphisms \eqref{Ab6e}
are defined by
\begin{displaymath} 
\left(
\begin{array}{cc}
X' &  J'\\
Z' & Y'
\end{array}
\right) \circ 
\left(
\begin{array}{cc}
X &  J\\
Z & Y
\end{array}
\right) :=
\left(
\begin{array}{cc}
X'X + \kappa(J')Z & X' J + J'Y\\
Z'X + Y'Z  & Z'\kappa(J) + Y'Y
\end{array}
\right).
\end{displaymath} 
Here the expression $X'J + J'Y$ makes sense because
$\mathcal{J}_{n+1}$ is a $\mathcal{W}_n(S)$-module.

\subsection{Lifting of truncated displays}

The following is the main result of this section.

\begin{prop}\label{Pr-lift}
Let $S \rightarrow R$ be a $pd$-thickening as above with kernel 
$\mathfrak{a}$. Let $m$ be a natural number such that  $p^m \mathfrak{a} = 0$.
Let $\bar{\mathcal{P}}_1$ and $\bar{\mathcal{P}}_2$ be  truncated displays of 
level $n$ over $R$. 
Let $\mathcal{P}_1$ resp. $\mathcal{P}_2$ be two relative truncated
displays of level $n$ for $S\to R$ which lift  $\bar{\mathcal{P}}_1$ resp.\ 
$\bar{\mathcal{P}}_2$. Then each morphism $\bar{\alpha}: \bar{\mathcal{P}}_1 
\rightarrow  \bar{\mathcal{P}}_2$ lifts to a morphism
\begin{equation}\label{Ab4e2} 
\alpha: \mathcal{P} \rightarrow \mathcal{P'}.
\end{equation} 

Assume moreover that $\bar{\mathcal{P}}_1$ and $\bar{\mathcal{P}}_2$ have 
order of nilpotence $\leq e$.
If $n > m(e+1) +1$, then the truncation 
\begin{displaymath} 
\alpha(n-m(e+1)-1): \mathcal{P}_1(n-m(e+1)-1) \rightarrow 
\mathcal{P}_2(n-m(e+1)-1)
\end{displaymath}
does not depend on the choice of $\alpha$ but only on 
$\bar{\alpha}$. 
\end{prop}

\begin{proof}
As in \cite{Zink-DFG} we may replace $\bar{\alpha}$ by the automorphism
$\left(\begin{smallmatrix}1&0\\\bar\alpha&1 \end{smallmatrix}\right)$
of $\bar{\mathcal{P}}_1 \oplus \bar{\mathcal{P}}_2$. 
Note that if $\bar{\mathcal P}_1$ and $\bar{\mathcal P}_2$ 
are nilpotent of order $\le e$, then
the same holds for $\bar{\mathcal{P}}_1 \oplus \bar{\mathcal{P}}_2$.
Thus is suffices to
prove the following assertion.

\emph{Let $\bar{\mathcal{P}}$ be a display of level $n$
over $R$ and let $\mathcal{P}$ and $\mathcal{P}'$ be two relative displays of
level $n$ for $S\to R$ which lift $\bar{\mathcal{P}}$. Then there is an isomorphism 
$\alpha: \mathcal{P} \rightarrow \mathcal{P}'$ which lifts the identity.
If $\bar{\mathcal{P}}$ is nilpotent of order $\le e$ then the truncation
$\alpha[n-m(e+1)-1]$ is uniquely determined.}
 
We choose a normal decomposition $\bar{P} = \bar{T} \oplus \bar{L}$.
For simplicity we
will assume that these modules are free with a given basis. Let $T$
and $L$ the free $\mathcal{W}_n(S)$-modules with basis which lift
$\bar{T}$ and 
$\bar{L}$. Then we have normal decompositions:
\begin{displaymath} 
P \cong T \oplus L, \quad P' \cong T \oplus L,
\end{displaymath} 
which lift the chosen normal decompositions of $\bar{\mathcal{P}}$. 
We are looking for homomorphisms of the form:
\begin{equation}\label{Ab7e}
\left(
\begin{array}{cc}
E_d & 0\\
0 & E_c
\end{array}
\right)
+
\left(
\begin{array}{cc}
X & J\\
Z & Y
\end{array}
\right): \mathcal{P} \longrightarrow \mathcal{P}'.
\end{equation}
The matrix $J$ has coefficients in $W_{n+1}(\mathfrak{a}) \subset
\mathcal{J}_{n+1}$ and the matrices $X, Y, Z$ have coefficients in the
kernel of $\mathcal{W}_n (S) \rightarrow \mathcal{W}_n (R)$. 

Let us describe this kernel. An element $\xi \in W_{n+1}(S)$ represents an
element of the kernel iff it takes the form $\xi = \eta +
~^{V^n}[s]$, where $\eta$ lies in $W_{n+1}(\mathfrak{a})$
and where $s \in S$ satisfies $ps \in \mathfrak{a}$. 
In this case the elements $\bar{\xi} = \Res \xi$ and $\sigma(\xi)$ lie in $W_n(\mathfrak{a})$, 
and the pairs $(\bar{\xi}, \sigma (\xi)) \in \mathcal{W}_n(S)$ for 
$\xi=\eta+~^{V^n}[s]$ as above are exactly
the elements in the kernel. 
We write the logarithmic coordinates of these elements with respect to
the divided powers on $\Fa$:
\begin{displaymath} 
\bar{\xi} = [a_0, \ldots, a_{n-1}], \qquad \sigma(\xi) = [x_1, \ldots,
x_n].
\end{displaymath} 
The logarithmic coordinates of $\sigma (~^{V^n}[s]) \in
W_n(\mathfrak{a})$ are $[0, \ldots, 0, ps]$. We see that $x_i = pa_i$
for $i \leq n-1$ and that $x_n \in pS \cap \mathfrak{a}$.
Thus the elements of the kernel correspond bijectively to vectors
\begin{displaymath} 
 <a_0, \ldots a_{n-1}, x_n>, \qquad a_i \in \mathfrak{a}, \quad x_n \in
 pS \cap \mathfrak{a}
\end{displaymath}
such that
\begin{displaymath} 
\begin{array} {l}
\sigma (<a_0, \ldots a_{n-1}, x_n>) = [pa_1, \ldots, pa_{n-1}, x_n]\\
\Res(<a_0, \ldots a_{n-1}, x_n>) = [a_0, \ldots a_{n-1}]. 
\end{array} 
\end{displaymath} 
With these notations we may write the matrices $X,Y,Z, J$:
\begin{displaymath} 
X = <X(0), \ldots, X(n)>
\end{displaymath} 
 where the $X(i)$ are matrices with
coefficients in $\mathfrak{a}$ and moreover $X(n)$ has coefficients in
$pS \cap \mathfrak{a}$ and similarly for $Y$ and $Z$. For the matrix
$J \in W_{n+1} (\mathfrak{a})$ we use the logarithmic coordinates
\begin{displaymath} 
J = [J(0), \ldots, J(n)],  \qquad \dot{\sigma}(J) = [J(1), \ldots,
J(n)]. 
\end{displaymath}      
The $J(i)$ are matrices with coefficients in $\mathfrak{a}$.

We assume that $\mathcal{P}$ and $\mathcal{P'}$ are given
by matrices as above \eqref{Ab5e}. We set
\begin{displaymath} 
 \left(
\begin{array}{cc}
\eta_A & \eta_B\\
\eta_C & \eta_D
\end{array}
\right) =
\left(
\begin{array}{cc}
A' & B'\\
C' & D'
\end{array}
\right)
-
\left(
\begin{array}{cc}
A & B\\
C & D
\end{array}
\right)
\end{displaymath} 

The condition that \eqref{Ab7e} is a homomorphism of relative displays
becomes:
\begin{displaymath} 
\left(
\begin{array}{cc}
\eta_A & \eta_B\\
\eta_C & \eta_D
\end{array}
\right)
+
\left(
\begin{array}{cc}
A' & B'\\
C' & D'
\end{array}
\right) 
\left(
\begin{array}{cc}
\sigma(X) & \dot{\sigma} (J)\\
p\sigma(Z)  & \sigma(Y)
\end{array}
\right)
= 
\left(
\begin{array}{cc}
\bar{X} & \bar{J}\\
\bar{Z} & \bar{Y}
\end{array}
\right)
\left(
\begin{array}{cc}
A & B\\
C & D
\end{array}
\right).
\end{displaymath} 
This is an equation in the $W_n(S)$-module $W_n(\mathfrak{a})$. We
rewrite it in logarithmic coordinates and obtain for $0 \leq i \leq
n-2$ the equations:
\begin{equation}\label{ABlog1e} 
\begin{array}{lcl} 
\left(
\begin{array}{cc}
\eta_A(i) & \eta_B(i)\\
\eta_C(i) & \eta_D(i)
\end{array}
\right)
& + &
\left(
\begin{array}{cc}
w_i(A') &w_i(B')\\
w_i(C' )& w_i(D')
\end{array}
\right)
\left(
\begin{array}{cc}
pX(i+1) & J(i+1)\\
p^2Z(i+1) & pY(i+1)
\end{array}
\right)
=\\[5mm]
& &
\left(
\begin{array}{cc}
X(i) & J(i)\\
Z(i) & Y(i)
\end{array}
\right)
\left(
\begin{array}{ll}
w_i(A) &w_i(B)\\
w_i(C )& w_i(D)
\end{array}
\right)
\end{array}
\end{equation} 
and for $i = n-1$ we obtain the equation
\begin{equation}\label{ABlog2e} 
\begin{array}{l} 
\left(
\begin{array}{cc}
\eta_A(n-1) & \eta_B(n-1)\\
\eta_C(n-1) & \eta_D(n-1)
\end{array}
\right)
 + 
\left(
\begin{array}{cc}
w_{n-1}(A') &w_{n-1}(B')\\
w_{n-1}(C' )& w_{n-1}(D')
\end{array}
\right)
\left(
\begin{array}{cc}
X(n) & J(n)\\
pZ(n) & Y(n)
\end{array}
\right) \\[5mm]
=
\left(
\begin{array}{cc}
X(n-1) & J(n-1)\\
Z(n-1) & Y(n-1)
\end{array}
\right)
\left(
\begin{array}{cc}
w_{n-1}(A) &w_{n-1}(B)\\
w_{n-1}(C )& w_{n-1}(D)
\end{array}
\right).
\end{array}
\end{equation} 
We see that for arbitrary given $X(n), Y(n), Z(n), J(n)$ there are unique solutions
of \eqref{ABlog2e} and of \eqref{ABlog1e} for $0\le i\le n$, which means that
for given $X(n), \ldots, J(n)$
there is a unique isomorphism \eqref{Ab7e} which lifts the identity. 

Assume now that $\bar\PPP$ is nilpotent of order $\le e$.
Let us write:
\[
H(i)=\left(\begin{matrix}X(i) & J(i) \\ Z(i) & Y(i) \end{matrix}\right).
\]
We claim that for each $k\ge 1$ and $0\le i\le n-k(e+1)-1$ the reduction
of $H(i)$ in $\Fa/p^k\Fa$ is independent of the choice of $H(n)$.
This proves the uniqueness assertion of the proposition.
Let
\[
\left(
\begin{array}{cc}
\breve{A} & \breve{B}\\
\breve{C} & \breve{D}
\end{array}
\right)
=
\left(
\begin{array}{cc}
A & B\\
C & D
\end{array}
\right)^{-1}.
\]
If we multiply \eqref{ABlog1e} by the image of this matrix under $w_i$,
we obtain for $0\le i\le n$ an equation
\[
H(i)=R(i)+
\left(\begin{matrix}w_i(A')&w_i(B')\\w_i(C')&w_i(D')\end{matrix}\right)
\left(\begin{matrix}1&0\\0&p\end{matrix}\right)
H(i+1)
\left(\begin{matrix}p&0\\0&1\end{matrix}\right)
\left(\begin{matrix}w_i(\breve A)&w_i(\breve B)\\w_i(\breve C)&w_i(\breve D)\end{matrix}\right)
\]
where $R(i)$ is a given matrix with coefficients in $\Fa$. Let
\begin{equation}
\label{Eq-Delta}
\Delta_i'=\left(\begin{matrix}w_i(A')&pw_i(B')\\
w_i(C')&pw_i(D')\end{matrix}\right),\qquad
\breve\Delta_i=\left(\begin{matrix}pw_i(\breve A)&pw_i(\breve B)\\
w_i(\breve C)&w_i(\breve D)\end{matrix}\right).
\end{equation}
Assume that $H(i)_{0\le i\le n}$ and $H'(i)_{0\le i\le n}$ are two solutions
of \eqref{ABlog1e} and \eqref{ABlog2e}.
For their difference $h=H-H'$ we obtain the equations
\begin{equation}
\label{Eq-h}
h(i)=\Delta_i'\cdot h(i+1)\cdot\breve\Delta_i
\end{equation}
for $0\le i\le n-2$.
If we can show that for $i\ge 0$ the product of $e+1$ factors
\[
\breve\Delta_{i+e}\cdot\ldots\cdot\breve\Delta_{i+1}\breve\Delta_i
\]
has coefficients in $pS$, it follows by induction that for $i\ge 0$ and $k\ge 1$ the
product of $k(e+1)$ factors
\[
\breve\Delta_{i+k(e+1)-1}\cdot\ldots\cdot\breve\Delta_{i+1}\breve\Delta_i
\]
has coefficients in $p^kS$. 
Then \eqref{Eq-h} implies that $h(i)=0$ for $i\le n-k(e+1)-1$,
which proves the claim.

Let $\breve{D}_0$ be the first component of the Witt vector matrix $\breve D$
modulo $p$. The assumption that $\bar{\mathcal P}$ is nilpotent of order
$\le e$ means that
\[
\breve D_0^{(p^{e-1})}  \cdot\ldots\cdot \breve{D}_0^{(p)} \breve{D}_0 \equiv 0
\]
modulo $\Fa(S/pS)$; see Lemma \ref{Le-nilpotent}.
But $a^p\in pS$ for $a \in \mathfrak{a}$ since $\mathfrak{a}$ has
divided powers. Thus we get
 \begin{displaymath} 
  \breve{D}_0^{(p^{e})}  \cdot \ldots \cdot \breve{D}_0^{(p^2)}
  \breve{D}_0^{(p)} = 0.
\end{displaymath}
Thus for $i\ge 0$ the lower right block of $\Delta_{i+e}\cdot\ldots\cdot\Delta_{i+1}$
has coefficients in $pS$, and it follows that all coefficients of
$\Delta_{i+e}\cdot\ldots\cdot\Delta_i$ lie in $pS$ as required.
This finishes the proof.
\end{proof}

\begin{cor} (Rigidity)\; 
Let $S \rightarrow R$, $\mathfrak{a}$ and $m \in
  \mathbb{N}$ be as in Proposition \ref{Pr-lift}. 

Let $\alpha_1, \alpha_2: \mathcal{P} \rightarrow \mathcal{P'}$ be two
morphisms of truncated displays of level $n$ over $S$. We denote the
truncated displays over $R$ which are obtained by base change with
$\bar{\mathcal{P}}$ and   $\bar{\mathcal{P}'}$ and assume that they
are nilpotent of order $\le e$. 

If the two morphisms $\bar{\alpha}_1,  \bar{\alpha}_2: 
\bar{\mathcal{P}} \rightarrow  \bar{\mathcal{P}'}$ agree, 
 we have
\begin{displaymath} 
\alpha_1[n-m(e+1)-1] = \alpha_2[n-m(e+1)-1] 
\end{displaymath} 
for the truncations.
\end{cor}

\begin{proof} 
Let $\tilde{\mathcal{P}}$ and  $\tilde{\mathcal{P}'}$ be
the relative truncated displays obtained from $\mathcal{P}$ and
$\mathcal{P}'$. It suffices to prove the equation of the Corollary
on the relative truncated displays. This is a statement of the Proposition. 
\end{proof} 

\subsection{The crystal of a nilpotent truncated display}

We explain how to associate to a nilpotent truncated display a
``truncated crystal''. 
We fix natural numbers $n$, $e$, and $m$ such that $n > m(e+1)+1$.
Let $R$ be a ring in which $p$ is nilpotent. 
Let $\Cris_m(R)$ be the category of 
all $pd$-thickenings $S \rightarrow R$ with kernel $\mathfrak{a}$ such that 
$p^{m}\mathfrak{a} = 0$.

Let $\mathcal{P}$ be a truncated display of level $n$ over $R$ which is
nilpotent of order $\leq e$.  We construct a locally free $S$-module 
$\mathbb{D}_{\mathcal{P}}(S)$ as follows. We choose a lifting of $\mathcal P$
to a relative 
truncated display $\tilde{\mathcal{P}}$ with respect to $S \rightarrow R$. 
Then we define 
\begin{equation}\label{AbCr1e}
\mathbb{D}_{\mathcal{P}}(S) = S \otimes_{\mathcal{W}_n(S)} \tilde{P}
\end{equation}
where the tensor product is taken with respect to the projection 
$\mathcal{W}_n(S) \rightarrow S$. In terms of the truncation of level $1$ of 
$\tilde{\mathcal{P}}$ we may write the right hand side of \eqref{AbCr1e} as
$S \otimes_{\mathcal{W}_1(S)}\tilde{P}[1]$. 
Therefore Proposition \ref{Pr-lift} shows that $\mathbb{D}_{\mathcal{P}}(S)$ 
does not depend on the choice of 
$\tilde{\mathcal{P}}$ and that $\mathbb{D}_{\mathcal{P}}(S)$ 
is functorial in $\mathcal{P}$. If $S_1 \rightarrow S_2$ is a morphism in 
$\Cris_m(R)$ we obtain a canonical isomorphism 
\begin{displaymath}
S_2 \otimes_{S_1} \mathbb{D}_{\mathcal{P}}(S_1) \cong \mathbb{D}_{\mathcal{P}}(S_2).
\end{displaymath}

Let $\bar{\mathcal{P}}$ be a truncated display of level $n$ over $R$, 
which is not necessarily nilpotent. Let 
$\mathcal{P}=(P,Q,\iota,\epsilon,F,\dot F)$ 
be a relative truncated display which lifts $\bar{\mathcal{P}}$. 
We consider the Hodge filtration of $\bar{\mathcal{P}}$:
\begin{displaymath}
R \otimes_{\mathcal{W}_{n}(R)} \bar{P} = \bar{P}/\kappa(I_{n+1}) \bar{P} 
\rightarrow \bar{P}/\iota(\bar{Q}).
\end{displaymath}
This homomorphism can be identified with the Hodge filtration of $\mathcal P$:
\begin{displaymath}
R \otimes_{\mathcal{W}_{n}(S)} P = P/\kappa(\mathcal{J}_{n+1})P \rightarrow 
P/\iota(Q).
\end{displaymath}
We define a lift of the Hodge filtration of $\PPP$ to $S$ as a commutative diagram:
\begin{displaymath}
\begin{CD}
S \otimes_{\mathcal{W}_{n}(S)} P @>>> \breve{T}\\
@VVV   @VVV\\
R \otimes_{\mathcal{W}_{n}(S)} P @>>> \bar{P}/\iota(\bar{Q}),
\end{CD}
\end{displaymath}
where $\breve{T}$ is a finitely generated projective $S$-module which lifts
$\bar P/\iota(\bar Q)$. 

We consider a truncated display $\mathcal{P}'$ of level $n$ whose
associated relative display is $\mathcal{P}$. 
We have $\mathcal P'=(P,Q',\iota,\epsilon,F,\dot F)$ with $Q'\subseteq Q$.
The Hodge filtration of $\mathcal P'$ is a lift of the Hodge filtration of $\mathcal P$, 
and we get back $Q'$ as the kernel of 
\[
Q\xrightarrow\iota S\otimes_{\mathcal W_n(S)}P\to P/Q'.
\]

Conversely, let $\breve T$ be a lift of the Hodge filtration of
$\mathcal P$ to $S$. Then we construct a truncated display $\mathcal{P}'$
as above whose Hodge filtration coincides with $\breve{T}$. 

Let $Q'$ be the kernel of 
\[
Q\xrightarrow\iota S\otimes_{\mathcal W_n(S)}P\to\breve T.
\]
We claim that we obtain a truncated display
$\mathcal P'=(P,Q',\iota,\epsilon,F,\dot F)$  of level $n$ over $S$.
It is easy to see that the restriction of $\epsilon:\mathcal J_{n+1}\otimes P\to Q$ 
to $I_{n+1,S}\otimes P$ lies in $Q'$, using that $\iota\circ\epsilon$ is the 
multiplication map.
Let $T\subseteq P$ and $L\subseteq Q$ be direct summands which give a 
normal decomposition of $\mathcal P$, which means that
$P\cong L\oplus T$ and $Q\cong\mathcal J_{n+1}\otimes T\oplus L$.
The composition $L\to P\to\breve T$ induces a homomorphism
$L\to\mathfrak{a}\breve T$, which we lift to $\phi: L\to W_{n+1}(\Fa)T$,
for example using the inclusion $\Fa\subset W_{n+1}(\Fa)$ by
the first logarithmic coordinate. If we replace the inclusion $i:L\to Q$
by $i-\phi$, then $L$ and $T$ define a normal decomposition for $\mathcal P'$.
The remaining axioms for truncated displays for $\mathcal P'$ follows easily.
Thus we have shown that lifts of $\mathcal P$ to truncated displays of
level $n$ over $S$ correspond to lifts of the Hodge filtration.

\begin{prop}\label{Ab3p}
Let $\bar{\mathcal{P}}$ be a truncated display of level $n$ over $R$ 
which is nilpotent of order $\le e$. Let $S\to R$ be a divided power
extension in $\Cris_m(R)$ with $n > m(e+1) + 1$. 
Then the isomorphism classes of liftings of $\bar{\mathcal{P}}$ 
to a truncated display of level $n$ over $S$ correspond bijectively to
liftings of the  
Hodge filtration of $\bar {\mathcal P}$ to $\mathbb{D}_{\bar{\mathcal{P}}}(S)$
as in the following diagram:
\begin{displaymath}
\hspace{-6.5em}
\xymatrix@M+0.2em{
\mathbb{D}_{\bar{\mathcal{P}}}(S) \ar[r] \ar[d] &
\breve{T} \ar[d] \\
\hspace{6.5em}\mathbb{D}_{\bar{\mathcal{P}}}(R) = R
\otimes_{\mathcal{W}_{n}(R)} \bar{P} \ar[r] & \bar{P}/\iota(\bar{Q})
}
\end{displaymath}
\end{prop}

\begin{proof} 
Let $\tilde{\mathcal{P}}$ be a lifting of $\bar{\mathcal{P}}$ to $S$, and
let $\tilde{\mathcal{P}}^{rel}$ be the associated relative truncated display for $S\to R$. 
By definition we have a well-defined isomorphism 
$\mathbb D_{\bar{\mathcal P}}(S)\cong S\otimes_{\mathcal W_n(S)}\tilde P$.
Thus the Hodge filtration of $\tilde P$ gives a lift of the Hodge filtration
as in the proposition. We obtain a map from the set of isomorphism classes
of liftings of $\bar{\mathcal P}$ to $S$ to the set of liftings of the Hodge filtration.
Since all liftings of $\bar{\mathcal P}$ to a relative
truncated display for $S\to R$ are isomorphic, the preceding considerations
show that this map is bijective.
\end{proof} 
We note that this Proposition gives only a bijection of isomorphism
classes. The bijection does not arise from an equivalence of categories.

\subsection{Lifting of displays}

Let $S\to R$ be a divided power extension
of rings in which $p$ is nilpotent with kernel $\Fa\subset S$. 
We want to see what the proof of Proposition \ref{Pr-lift} 
gives for non-truncated displays. 

We recall the definition of relative displays.
Let $\JJJ_{S/R}$ be the kernel of $W(S)\to R$. 
Let $\dot\sigma:I_S\to W(S)$ be the inverse of the
Verschiebung and let $\dot\sigma:\JJJ_{S/R}\to W(S)$ extend this
map by $\dot\sigma(x)=0$ if $x\in W(\Fa)$ is an element with 
logarithmic coordinates $[a,0,0,\ldots]$.
A relative display for $S\to R$ consists of $(P,Q,F,\dot F)$
where $Q\subseteq P$ are $W(S)$ modules and where
$F:P\to P$ and $\dot F:Q\to P$ are $\sigma$-linear maps such that

\smallskip
(i) {}
$P$ is a finitely generated projective $W(S)$-module,

\smallskip
(ii) {}
$\JJJ_{S/R}P\subseteq Q$ and $P/Q$ is a projective $R$-module,

\smallskip
(iii) {}
$\dot F(ax)=\dot\sigma(a)F(x)$ for $a\in\JJJ_{S/R}$ and $x\in P$,

\smallskip
(iv) {}
$\dot F(Q)$ generates $P$.

\begin{prop}
\label{Pr-lift-disp}
Let $\PPP_1$ and $\PPP_2$ be two relative displays for $S\to R$
and let $\bar\PPP_1$ and $\bar\PPP_2$ be their reductions to 
displays over $R$. We consider the reduction map
\[
\rho:\Hom(\PPP_1,\PPP_2)\to\Hom(\bar\PPP_1,\bar\PPP_2).
\]
Then the following hold.

\smallskip
\noindent
(a) {} If $\Fa$ is an $S$-module of finite length, the map $\rho$ is surjective,

\smallskip
\noindent
(b) {} If $\PPP$ and $\PPP'$ are nilpotent, the map $\rho$ is bijective.
\end{prop}

Assertion (b) is proved in \cite[Theorem 44]{Zink-DFG}.
We recall it here for completeness.

\begin{proof}
By passing to $\PPP_1\oplus\PPP_2$ it suffices to prove the following assertion.

\emph{Let $\bar\PPP$ be a display over $R$ and let $\PPP$ and $\PPP'$ be
two relative displays for $S\to R$ which lift $\bar\PPP$. If $\Fa$ is an $S$-module
of finite length then there is an isomorphism $\PPP\cong\PPP'$ which lifts
the identity. If $\bar\PPP$ is nilpotent then there is a unique isomorphism
$\PPP\cong\PPP'$.}

We can assume that $\PPP$ and $\PPP'$ have the same underlying
modules with normal decomposition $P=T\oplus L$ and $Q=\JJJ_{S/R}T\oplus L$.
For simplicity we assume that $T$ and $L$ are free, $T=W(S)^d$ and $L=W(S)^c$.
Then $\PPP$ and $\PPP'$ are given by matrices
$\left(\begin{smallmatrix}
A&B\\C&D
\end{smallmatrix}\right)$
and
$\left(\begin{smallmatrix}
A'&B'\\C'&D'
\end{smallmatrix}\right)$
with difference in $W(\Fa)$ :
\[
\left(\begin{matrix}
\eta_A&\eta_B\\\eta_C&\eta_D
\end{matrix}\right)=
\left(\begin{matrix}
A'&B'\\C'&D'
\end{matrix}\right)-
\left(\begin{matrix}
A&B\\C&D
\end{matrix}\right)
\]
The components of $\eta_A$ with respect to the isomorphism
$\log:W(\Fa)\cong\Fa^\infty$ are denoted by $\eta_A(n)$ for $n\ge 0$.
These are matrices with coefficients in $\Fa$.
As in \eqref{Ab7e} we are looking for matrices
\[
H(n)=\left(\begin{matrix}X(n) & J(n) \\ Z(n) & Y(n) \end{matrix}\right)
\]
for $n\ge 0$ with coefficients in $\Fa$ such that for $n\ge 0$ we have
\[
H(n)=R(n)+\Delta_n'H(n+1)\breve\Delta_n
\]
where $\breve\Delta_n$ and $\Delta_n'$ are defined as in \eqref{Eq-Delta}.
Here $R(n)$ for $n\ge 0$ are given matrices with coefficients in $\Fa$.
Let $\HHH=M_{c+d}(\Fa)$. For $n\ge 0$ let $U_n:\HHH\to\HHH$
be the map $U_n(H)=\Delta_n'H\Delta_n$. Then a solution $H(n)_{n\ge 0}$
exists if and only if the image of $R(n)_{n\ge 0}$ in
\[
{\varprojlim}^1(\HHH\xleftarrow{U_0}\HHH\xleftarrow{U_1}\HHH\xleftarrow{U_2}\cdots)
\]
is zero, and the solution is unique if and only if the $S$-module
\[
{\varprojlim}(\HHH\xleftarrow{U_0}\HHH\xleftarrow{U_1}\HHH\xleftarrow{U_2}\cdots)
\]
is zero. If $\Fa$ is an $S$-module of finite lenght, then $\HHH$ has finite
length. Thus $\varprojlim^1$ is zero by the Mittag-Leffler condition.
Assume that $\bar\PPP$ is nilpotent. Let $p^k\Fa=0$.
We saw in the proof of Proposition \ref{Pr-lift} that for each $n\ge 0$,
the product
$\breve\Delta_{n+k(e+1)-1}\cdot\ldots\cdot\breve\Delta_n$ has coefficients in $p^kS$.
Thus 
\[
U_{n+k(e+1)-1}\circ\ldots\circ U_n=0,
\]
which implies that $\varprojlim$ and $\varprojlim^1$ are zero.
\end{proof}

\section{Truncated $p$-divisible groups and displays}

\subsection{The functor from groups to displays}

Let $S \rightarrow R$ a $pd$-thickening with kernel $\mathfrak{a}$. We assume 
that the divided powers on $\mathfrak{a}$ are compatible with the canonical
divided powers on $pS$.
In \cite{Lau-Smoothness} one of us has defined a functor 
\begin{displaymath}
\Phi_{S/R}:\BBB\TTT(R) \rightarrow \mathcal{D}(S/R)
\end{displaymath}
from the category $p$-divisible groups over $R$ to the category of displays
relative to $S \rightarrow R$, and in the case $pR=0$ also a functor
\begin{equation}\label{L2functor1e}
\Phi_{n,R}:\BBB\TTT_n(R)\to\DDD_n(R)
\end{equation} 
from the category of truncated $p$-divisible groups over $R$ of level $n$
to the category of truncated displays over $R$ of level $n$.
We will indicate a few modifications 
to adapt this to the notion of truncated displays which we use here
and to the case of relative truncated displays. 

We will use the derived category $D^{\flat}(\mathcal{E})$ of an exact category 
$\mathcal{E}$. Let $A^{\flat}(\mathcal{E})$ be the full subcategory of the bounded
homotopy category $K^{\flat}(\mathcal{E})$ which consists of all complexes 
which split into short exact sequences in the sense of $\mathcal{E}$. 
By a result of \cite{Thomason}, 1.11.6 (see also \cite{Neeman}), 
$A^{\flat}(\mathcal{E})$ is \'epaisse as a subcategory 
if each idempotent $e: E \rightarrow E$ which factors as
$E \overset{\alpha}{\rightarrow} F \overset{\beta}{\rightarrow} E$ such that 
$\alpha \circ \beta = \id_{F}$, is a split idempotent. In this case the 
localization of $K^{\flat}(\mathcal{E})$ with respect to
$A^{\flat}(\mathcal{E})$ defines the derived  category $D^{\flat}(\mathcal{E})$.  

Let $\mathcal{G}$ be the category of finite locally free group schemes over 
$R$ which admit an embedding in a $p$-divisible group. We consider the bounded
derived category $D^{\flat}(\BBB\TTT(R))$ of $p$-divisible
groups. (Note that every idempotent in $\BBB\TTT(R)$ is split.) 
Writing an object of $G \in \mathcal{G}$ as 
a kernel of an isogeny of $p$-divisible groups we obtain a fully faithful 
functor
\begin{displaymath}
\mathcal{G} \rightarrow D^{\flat}(\BBB\TTT_R)
\end{displaymath}
The image of this functor lies in the full subcategory 
$D^{\flat}_{\leq 1}(\BBB\TTT_R)$ 
generated by complexes $X^{\cdot}$ of $p$-divisible groups for which $X^{i} = 0$ 
for $i \geq 2$. One can check that a morphism $X^{\cdot} \rightarrow Y^{\cdot}$
in the category $D^{\flat}_{\leq 1}(\BBB\TTT_R)$ may be represented by morphisms of 
complexes 
\begin{displaymath}
X^{\cdot} \leftarrow Z^{\cdot} \rightarrow Y^{\cdot},
\end{displaymath}
where $Z^{i} = 0$ for $i \geq 2$ and where the left arrow is a 
quasiisomorphism. 

Let us formalise the linear data of (relative) truncated displays.
Let $A$ be a ring and let 
$\kappa: \mathfrak{c}  \rightarrow A$ be a homomorphism
of $A$-modules such that for $x,y\in\mathfrak c$ we have
$$
\kappa(x)y=\kappa(y)x.
$$
Main example: Consider a surjective ring homomorphism $B\to A$ 
and an ideal $\mathfrak c\subset B$ which is an $A$-module,
i.e. annihilated by the kernel of $B \to A$. 
We take $\kappa:\mathfrak c\to B\to A$.

\begin{df}
An $(A,\mathfrak{c})$-module consists of 
$(M, N, \iota, \epsilon)$, where $M$ and $N$ are $A$-modules and $\iota$ and 
$\epsilon$ are $A$-module homomorphisms
\begin{displaymath}
\mathfrak{c} \otimes_{A} M \overset{\epsilon}{\rightarrow} N 
\overset{\iota}{\rightarrow} M,
\end{displaymath}
such that the composition of these two maps is the multiplication, i.e.\ 
$c \otimes m$ is mapped to $\kappa(c)m$, and such that the composition 
of the following maps is the multiplication:
\begin{displaymath}
\mathfrak{c} \otimes_{A} N \overset{\id \otimes \iota}{\longrightarrow }
\mathfrak{c} \otimes_{A} M \overset{\epsilon}{\rightarrow} N.
\end{displaymath}
\end{df}

The category of $(A,\mathfrak{c})$-modules is in the obvious way abelian.

If $\mathcal{P} = (P,Q,\iota, \epsilon, F, \dot{F})$ is a truncated display
over $R$ (resp.\ relative to $S \rightarrow R$) of level $n$, then 
$(P,Q,\iota, \epsilon)$ is a $(\mathcal{W}_n(R), I_{n+1})$ (resp. 
$(\mathcal{W}_n(S), \mathcal{J}_{n+1})$)-module. 
 
We call a sequence of truncated or relative truncated displays exact 
if the underlying sequence of $(\mathcal{W}_n(R),I_{n+1})$-modules
or $(\mathcal{W}_{m}(S),\mathcal{J}_{m+1})$-modules is exact. 
We obtain exact categories in which every idempotent is split; note
that all defining properties of (relative) truncated displays pass
over to direct summands. Thus again the bounded derived categories
exist:  
\begin{displaymath}
D^{\flat}(\mathcal{D}_m(S/R)), \quad D^{\flat}(\mathcal{D}_m(R)).
\end{displaymath}
Similarly we have the bounded derived categories 
$D^{\flat}(\mathcal{D}(R))$ and 
$D^{\flat}(\mathcal{D}(S/R))$ of displays and of relative displays.
For each natural number $m$ we obtain functors: 
\begin{equation}\label{Eq-GDbDm}
\mathcal{G} \rightarrow D^{\flat}_{\leq 1}(\BBB\TTT(R)) \xrightarrow{\Phi_{S/R}} 
D^{\flat}_{\leq 1}(\mathcal{D}(S/R)) \rightarrow 
D^{\flat}_{\leq 1}(\mathcal{D}_m(S/R)). 
\end{equation}

We  consider the category $\mathcal T_m=\mathcal{T}_m(S/R)$ of all data 
$(P,Q,\iota, \epsilon, F, \dot{F})$ as in the definition of a relative 
truncated display of level $m$, but we no longer require the conditions
(i), (iii), and (vi) of Definition \ref{Ab2d}. Let 
\begin{displaymath}
(P_1,Q_1,\iota_1, \epsilon_1, F_1, \dot{F}_1) \rightarrow 
(P_2,Q_2,\iota_2, \epsilon_2, F_2, \dot{F}_2)
\end{displaymath}
be a morphism in $\mathcal{T}_m$. Then one easily defines a cokernel 
$(P,Q,\iota, \epsilon, F, \dot{F})$ such that $(P,Q,\iota, \epsilon)$ is the
cokernel in the category of $(\mathcal{W}_n(S),
\mathcal{J}_{n+1})$-modules. 
For $G\in\GGG$ we apply \eqref{Eq-GDbDm} and take $H^1$, which 
is a cokernel.  In this way, for each natural number $m$ we obtain a functor:
\begin{equation}\label{AbLF1e}
\Phi_m:\mathcal{G} \rightarrow \mathcal{T}_m.
\end{equation}

\begin{prop}\label{AbLF1p}
Let $S \rightarrow R$ be a $pd$-thickening and assume that $p$ is nilpotent 
in $S$. Let $n \geq m$ be natural numbers such that $p^{n}W_{m}(S) 
= 0$. Then \eqref{AbLF1e} induces a functor
\begin{displaymath}
\Phi_m:
\BBB\TTT_n(R) \rightarrow \mathcal{D}_m (S/R)
\end{displaymath}
from the category of truncated $p$-divisible groups of level $n$ over $R$
to the category of truncated relative displays of level $m$ for $S\to R$.
\end{prop}

Note that in the case $pS=0$ we can take $n=m$.

\begin{proof} 
The condition $p^{n}W_{m}(S) = 0$ is equivalent with
$p^{n}\mathcal{W}_{m}(S) = 0$. 
 
Let $\mathcal{G}_n$ the full subcategory of $\mathcal{G}$ 
which consists of truncated $p$-divisible groups. 
We claim that the functor $\Phi_m$ of \eqref{AbLF1e} restricted to $\GGG_n$
takes values in the category of relative truncated displays of level $m$,
i.e.\ that the conditions (i), (iii), and (vi) of Definition \ref{Ab2d}
are satisfied.

For $G\in\GGG$ let $\Phi_m(G)=(P,Q,\iota,\epsilon,F,\dot F)$.
Here $P$ is a $\WWW_m(S)$-module of finite presentation. 
$\Coker(\iota)$ and $\Coker(\epsilon)$ are $R$-modules of
finite presentation. These modules are compatible with
base change under homomorphisms of $pd$-thickenings from
$S\to R$ to $S'\to R'$. 

By \cite[Th\'eor\`eme 4.4]{Illusie-BT} we may lift a truncated $p$-divisible group over $R$ to
a truncated $p$-divisible group over $S$. Zariski locally on $\Spec S$,
the lifted group can be embedded into a $p$-divisible group.
Therefore, to prove the claim we may assume that $R=S$.

If $X$ is a $p$-divisible group over $R$ and if $X(n)$ is its truncation,
we can use the resolution $0 \rightarrow X(n) \rightarrow X 
\overset{p^{n}}{\rightarrow } X$ to compute $\Phi_m(G)$.
Since we assumed that $p^n\WWW_m(S)=0$, we get that
$\Phi_m(G)$ is the $m$-truncation of the display associated to $X$.
So the claim is proved in this case.

If $R$ is a noetherian complete local ring with perfect
residue field, each truncated $p$-divisible group $G$ of level $n$
over $R$ takes the form $X(n)$ for a $p$-divisible group $X$
by \cite[Th\'eor\`eme 4.4]{Illusie-BT}. So the claim holds over $R$.

If $R$ is a noetherian ring, for each prime ideal $\Fp$
of $A$ we find a faithfully flat ring homomorphism $\hat A_\Fp\to A'$
where $A'$ is a noetherian complete local ring with perfect (or algebraically closed)
residue field. 
By descent (see Corollary \ref{Cor:descent-module}) it follows that for $G\in\GGG_n$,
$\Phi_m(G)$ satisfies the conditions (i), (iii), and (vi) of Definition \ref{Ab2d}.
This proves the claim when $R$ is noetherian. 

Since a truncated $p$-divisible group embeds Zariski locally in a
$p$-divisible group, if $R$ is noetherian we can extend the functor
$\Phi_m$ from $\mathcal{G}_n$ to all truncated $p$-divisible groups by
descent
of truncated displays (see Proposition \ref{Pr:descent-disp}).
Finally we can use base change to define the functor $\Phi_m$
over a base which is not noetherian.
\end{proof}

For $S=R$ with $pR=0$ and $n=m$, 
the functor $\Phi_n$ of Proposition \ref{AbLF1p} is \eqref{L2functor1e}.
We note that this functor preserves the order of nilpotence:

\begin{lemma}
\label{Le:nilpotence-Phi}
Let $R$ be a ring with $pR=0$.
A truncated $p$-divisible group $G$ over $R$ of level $n\ge 1$
has nilpotence order $e$ iff the associated truncated display
$\PPP=\Phi_n(G)$ has nilpotence order $e$.
\end{lemma}

\begin{proof}
The construction of $\Phi_n(G)$ gives an isomorphism
$\Lie(G^\vee)\cong \iota(Q)/I_nP$ such that the
Verschiebung $V$ of $G^\vee$ corresponds to the
homomorphism $\bar V^\sharp$ of \eqref{AbVis3e}.
Then the lemma follows from the proof of Lemma \ref{Le:nilpotence}.
\end{proof}

\begin{prop}
\label{Pr:DXDP}
Let $t\ge n$ such that $p^tW_n(R)=0$.
Let $X$ be a $p$-divisible group over $R$ such that $X_{R/pR}$
is nilpotent of order $\le e$, and let $\PPP=\Phi_n(X(t))$.
If $S\to R$ is an object of $\Cris_m(R)$ with $n>m(e+1)+1$, we have
a canonical isomorphism
\begin{equation}
\label{Eq:DXDP}
\mathbb D_X(S)\cong\mathbb D_\PPP(S)
\end{equation}
between the Grothendieck-Messing crystal of $X$ and the crystal
of $\PPP$, evaluated at $S\to R$.
\end{prop}

\begin{proof}
Let $\PPP_X=\Phi_R(X)$ be the display of $X$ and $\tilde\PPP_X=\Phi_{S/R}(X)$
the display relative to $S\to R$ associated to $X$. 
Since $\tilde\PPP_X(n)$ is a lift of $\PPP_X(n)=\PPP$, by \eqref{AbCr1e}
the right hand side of \eqref{Eq:DXDP} is $S\otimes_{W(S)}\tilde P$.
By the construction of the functor $\Phi_{S/R}$, this module also coincides with the 
left hand side of \eqref{Eq:DXDP}.
\end{proof}

Later we will use the following consequence.

\begin{cor}
\label{Co:DefXXt}
Let $S$ be a ring with $p^{m+1}S=0$.
Let $X$ be a $p$-divisible group over $R=S/pS$ which is nilpotent of order $e$.
For $t\ge m(e+2)+2$ the set of isomorphism classes of lifts of $X$ to $S$ is
bijective to the set of  isomorphism classes of lifts of $X(t)$ to $S$.
\end{cor}

\begin{proof}
Let $n=m(e+1)+2$. Then $t\ge n+m$ and thus $p^tW_n(S)=0$, using that
$p^nW_n(R)=0$ and $p^mW_n(pS)=0$. Let  $\PPP=\Phi_n(X(t))$. We have two maps
\[
\Def_{S/R}(X)\to \Def_{S/R}(X(t))\xrightarrow{\Phi_n} \Def_{S/R}(\PPP)
\]
where $\Def_{S/R}$ means set of isomorphism classes of lifts to $S$.
By \cite[Th\'eor\`eme 4.4]{Illusie-BT}, the first map is surjective. 
Propositions \ref{Pr:DXDP} implies that the composition is bijective,
using that deformations of $X$ and of $\PPP$ are both classified by 
lifts of the Hodge filtration, by Proposition \ref{Ab3p} and by the Grothendieck-Messing Theorem.
It follows that the first map is bijective.
\end{proof}

\subsection{Exactness and Duality}

Let us return for a moment to the study of $(A,\mathfrak c)$-modules.
Let $\mathfrak c\to A$ be as above.
Let $\mathfrak{u}$ be the kernel of $\kappa:\mathfrak{c} \rightarrow A$
and let $R$ be its cokernel,
$$
0\to\mathfrak u\to\mathfrak c\to A\to R\to 0.
$$
Here $R$ is a factor ring of $A$ by the ideal $\kappa(\mathfrak{c})$,
and $\mathfrak u$ is an $R$-module 
since we have $\kappa(\mathfrak c)\mathfrak u=\kappa(\mathfrak
u)\mathfrak c=0$. We assume in the following that all finitely
generated projective $R$-modules lift to finitely generated projective
$A$-modules.  

For given finitely generated projective $A$-modules $T$ and $L$ we define an 
$(A,\mathfrak{c})$-module $\mathcal{S}(T,L)$ as follows: We set
\begin{displaymath}
M = T \oplus L, \quad N = \mathfrak{c} \otimes_{R} T \oplus L,
\end{displaymath}
with the obvious maps $\iota$ and $\epsilon$; see Definition \ref{AbND1d}.
An $(A,\mathfrak{c})$-module which is isomorphic to some $\mathcal{S}(T,L)$ 
will be called standard projective, and $M = T \oplus L$ is called a normal 
decomposition.  

For an arbitrary 
$(A,\mathfrak{c})$-module $(M, N, \iota, \epsilon)$ we have a canonical 
isomorphism
\begin{displaymath}
\Hom (\mathcal{S}(T,L), (M, N, \iota, \epsilon)) = \Hom_{R}(T,M) \oplus 
\Hom_R(L,N).
\end{displaymath}
In particular $\mathcal{S}(T,L)$ is a projective object in the category of 
$(A,\mathfrak{c})$-modules (this does not need that $T$ and $L$ are finitely generated). 
Obviously we have projective resolutions in this category. 

To each $(A,\mathfrak{c})$-module $(M,N,\iota,\epsilon)$ we associate the following
complex of $A$-modules:
\begin{equation}\label{Ab4term1e}
0 \rightarrow \mathfrak{u}\otimes_{A} (M/N) \xrightarrow{\epsilon'} 
N \rightarrow M \rightarrow M/N \rightarrow 0 
\end{equation}
where $\varepsilon'$ is the restriction of $\epsilon$.

\begin{lemma}
\label{Le-std-proj}
An $(A,\mathfrak c)$-module $\breve M=(M,N,\iota,\epsilon)$ is standard projective iff the
following holds.
\begin{enumerate}
\item[(i)]
$M$ is a finitely generated projective $A$-module,
\item[(ii)]
$M/N$ is a finitely generated projective $R$-module,
\item[(iii)]
the sequence \eqref{Ab4term1e} is exact.
\end{enumerate}
\end{lemma}

\begin{proof}
(Cf.\ \cite[Lemma 3.3]{Lau-Smoothness})
Clearly standard projective modules satisfy (i)-(iii).
Assume that (i)-(iii) hold. Since $\Image\epsilon'\subset\Image\epsilon$,
the exact sequence \eqref{Ab4term1e} implies that the following is exact:
$$
0\to N/\Image \epsilon\to M/\mathfrak c M\to M/N\to 0
$$
Thus $N/\Image\epsilon$ is a finitely generated projective $R$-module.
Let $T$ and $L$ be finitely generated projective $A$-modules which lift
$M/N$ and $N/\Image\epsilon$. 
We have a homomorphism $g:\mathcal S(T,L)\to\breve M$,
and the associated homomorphism of exact sequences \eqref{Ab4term1e}
is an isomorphism on all components, except possibly on $N$. 
By the 5-Lemma $g$ is an isomorphism.
\end{proof}

\begin{lemma}
\label{Le-std-proj-ker}
Let $0\to\breve M_1\to\breve M_2\to\breve M_3\to 0$ be a short exact
sequence of $(A,\mathfrak c)$-modules.
If $\breve M_2$ and $\breve M_3$ are standard projective, then so is $\breve M_1$.
\end{lemma}

\begin{proof}
We write $\breve M_i=(M_i,N_i,\iota_i,\epsilon_i)$.
Clearly $M_1$ is finitely generated projective over $A$.
Consider the commutative diagram with exact rows:
\begin{displaymath}
\begin{CD}
0 @>>> N_1 @>>> N_2 @>>> N_3 @>>> 0\\
@. @VVV  @VVV @VVV @.\\
0 @>>> M_1 @>>> M_2 @>>> M_3 @>>> 0
\end{CD}
\end{displaymath}
Applying the snake-lemma and taking into account the exact 
sequences \eqref{Ab4term1e} for $\breve{M}_2$ and $\breve{M}_3$ we obtain
an exact sequence of projective $R$-modules:
\begin{displaymath}
0 \rightarrow M_1/N_1 \rightarrow M_2/N_2 \rightarrow M_3/N_3 \rightarrow 0.
\end{displaymath}
In particular $M_1/N_1$ is finitely generated projective over $R$.
Since the last sequence remains exact under $\mathfrak u\otimes_R{}\;$,
it follows that \eqref{Ab4term1e} is exact for $\breve M_1$.
\end{proof}

\begin{prop}
\label{Abexact1p}
Let $\mathcal{P}_i = (P_i,Q_i,\iota_i, \epsilon_i, F_i, \dot{F}_i)$ for 
$i = 1,2$ be two truncated displays of level $n$ over a ring $R$.
Let $\alpha: \mathcal{P}_1 \rightarrow \mathcal{P}_2$ be a morphism 
such that $P_1 \rightarrow P_2$ and $Q_1\to Q_2$ are surjective.

Then there is a truncated display of $\mathcal{P}$ level $n$ and a sequence
of truncated displays
\begin{displaymath}
0 \rightarrow \mathcal{P} \rightarrow \mathcal{P}_1 \rightarrow \mathcal{P}_2
\rightarrow 0
\end{displaymath}
such that the underlying sequence of $(\mathcal{W}(R), I_{n+1})$-modules 
is exact. 

The same statement is true for relative truncated displays.
\end{prop}

Remark: One can also show that surjectivity of $P_1\to P_2$ implies
surjectivity of $Q_1\to Q_2$. 

\begin{proof}
We consider the case of relative truncated displays.
Let $\PPP=(P,Q,\iota,\epsilon,F,\dot F)$ be the kernel of $\alpha$,
taken componentwise. Then $(P,Q,\iota,\epsilon)$ is a standard
projective $(\mathcal W_n(S),\mathcal J_{n+1})$-module by Lemma \ref{Le-std-proj-ker}.
The underlying sequence of $(\mathcal W_n(S),\mathcal J_{n+1})$-modules
splits. Now $\PPP$ is a relative truncated display iff the operator
$F^\sharp\oplus\dot F^\sharp$ of \eqref{Ab4e} is an isomorphism.
Since a block upper triangular matrix is invertible iff the
diagonal blocks are invertible, the fact that $\PPP_1$ is a relative
truncated display implies the same for $\PPP$.
\end{proof}

\begin{cor}
\label{Co-Phi-ex}
The functor $\Phi_m$ of Proposition \ref{AbLF1p} is exact.
\end{cor}

\begin{proof}
A given short exact sequence $0\to G_1\to G_2\to G_3\to 0$ in $\BBB\TTT_n(R)$ embeds
Zariski locally into a short exact sequence of $p$-divisible groups $0\to X_1\to X_2\to X_3\to 0$.
Let $Y_i=X_i/G_i$. 
We have exact sequences in $\DDD_m(S/R)$ which define $\MMM_i$:
$$
0\to\MMM_i\to\tau_m\Phi_{S/R}(X_i)\to\tau_m\Phi_{S/R}(Y_i)\to\Phi_m(G_i)\to 0.
$$
Here $\tau_m$ means truncation to level $m$.
Indeed, the sequence without $\MMM_i$ is exact by definition.
Proposition \ref{Abexact1p} implies that the image and kernel 
of the middle arrow are relative truncated displays. 

By the snake lemma we obtain an exact sequence in $\DDD_m(S/R)$
$$
0\to\MMM_1\to\MMM_2\to\MMM_3\to\Phi_m(G_1)\to\Phi_m(G_2)\to\Phi_m(G_3)\to 0.
$$
Let $\Phi_m(G_i)=(P_i,\ldots)$. The rank of $P_i$ is the height of $G_i$.
Since the height is additive in short exact sequences,
it follows that $M_3\to P_1$ is the zero map. By Lemma \ref{Le-faithful-rel}
it follows that $\MMM_3\to\PPP_1$ is zero.
\end{proof}

\begin{remark}(Duality)
\label{Rk-dual}
Let $G$ be a truncated $p$-divisible group of level $n$ over $R$. 
Assume that $G$ is the kernel of an isogeny of $p$-divisible groups: 
$0 \rightarrow G \rightarrow X_0 \rightarrow X_1 \rightarrow 0$. 
Let $\alpha:\mathcal{P}_0 \rightarrow \mathcal{P}_1$ be the morphism of relative 
truncated displays of level $m$, given by the functor $\Phi_{S/R}$ followed by truncation. 
By construction, $\Phi_m(G)=\Coker\alpha$.
We claim that there is a natural isomorphism
$$
\Phi_m(G)\cong\Ker(\alpha),
$$
i.e.\ one could also define $\Phi_m$ using the kernel.
First we note that the kernel is a relative  truncated display
by Proposition \ref{Abexact1p}. Now we have an exact
sequence in $\BBB\TTT_n(R)$
$$
0\to G\to X_0(n)\to X_1(n) \to G\to 0.
$$
Since $\Phi_m$ is exact and since $\Phi_m(X_i(m))=\PPP_i$ the claim follows.

One can define the dual of (relative) truncated displays as in the case of
(relative) displays. The functor $\Phi_{S/R}$ preserves duality.
Using the above isomorphism one can deduce that the functor $\Phi_m$
preserves duality too. 
We leave the details to the reader.
\end{remark}

\subsection{Smoothness}
\label{Se:Smoothness}

The functors $\Phi_n$ over rings $R$ with $pR=0$ define a morphism
\[
\phi_n:\BBB\TTT_n\times\Spec\FF_p\to\DDD_n\times\Spec\FF_p
\]
of smooth algebraic stacks over $\FF_p$.
By \cite[Theorem 4.5]{Lau-Smoothness} this morphism is smooth.
Using Proposition \ref{Pr-lift-disp} we can simplify the proof.
This remark is independent of the notion of relative truncated displays.
Let $k$ be a field of characteristic $p$. We consider the ring
homomorphism $S=k[\varepsilon]\to R=k$.
To prove that $\phi_n$ is smooth it suffices to show that the 
morphism of fpqc stacks
\[
\phi:\BBB\TTT\to\DDD
\]
from $p$-divisible groups to displays satisfies the lifting criterion
of formal smoothness with respect to $S\to R$. We equip the
kernel of $S\to R$ with the trivial divided powers. We consider
the commutative diagram of functors:
\[
\xymatrix@M+0.2em{
\BBB\TTT(S) \ar[r]^-f \ar[d]_{\Phi_S} &
\BBB\TTT(R) \ar@{=}[r] \ar[d]^{\Phi_{S/R}} &
\BBB\TTT(R) \ar[d]^{\Phi_R} \\
\DDD(S) \ar[r]^-g &
\DDD(S/R) \ar[r]^-h &
\DDD(R) 
}
\]
Here the left hand square is 2-Cartesian because lifts under $f$ or
under $g$ correspond to lifts of the Hodge filtration. For $f$
this is the Grothendieck-Messing theorem, and for $g$ this is trivial.
Proposition \ref{Pr-lift-disp} (a) implies that for each display over $R$
all lifts under $h$ are isomorphic. The lifting criterion for $S\to R$
follows easily.

\subsection{From displays to groups } 
 
Let $R$ be a ring with $pR=0$.
We will view formal groups and group schemes with a nilpotent augmentation 
ideal as functors on the category $\Nil_{R}$ of nilpotent $R$-algebras. 
We will call such group schemes infinitesimal. 

Let $G$ be a functor on $\Nil_R$. We recall the definition of the Frobenius of $G$.
For $N \in \Nil_R$ we have the absolute Frobenius 
$\Frob^{n}: N \rightarrow N_{[p^{n}]}$. This induces a homomorphism

\begin{equation*}
\Frob_{G}^{n}: G(N) \rightarrow G^{(p^{n})}(N) = G(N_{(p^{n})})
\end{equation*}
which is called the Frobenius of $G$.
We denote by $G[F^n]$ the kernel of $\Frob_{G}^{n}$. Let $N'$ be the kernel of 
$\Frob^{n}: N \rightarrow N_{[p^{n}]}$. 
If $G$ is left exact we have $$G(N') =  G[F^{n}](N) = G[F^{n}](N').$$ 
Let $\Nil_R^{(n)}\subset\Nil_R$ be the category of $R$-algebras $N$ such that 
$x^{p^n}=0$ for all $x\in N$. For a left exact functor $G$ on $\Nil_R$
we can view $G[F^{(n)}]$ 
as the restriction of $G$ to the category $\Nil_R^{(n)}$.

If $G$ is a commutative formal group of dimension $d$ then $G[F^n]$ is a 
finite locally free infinitesimal group scheme of rank $p^{dn}$ over $R$. 
Finite locally free infinitesimal group schemes which arise in this way are 
called truncated formal groups of level $n$ over $R$. 
Let $\FFF\GGG_n(R)$ be the category of such group schemes.

Let $\mathcal{P} =(P,Q,\iota,\epsilon,F,\dot F)$ 
be a truncated display of level $n$ over $R$. 
We chose a normal decomposition 
\[
P=T\oplus L,\qquad Q=I_{n+1}\otimes T\oplus L.
\]
Let $N\in\Nil_R^{(n)}$.
Then the $W(R)$-module $W(N)$ is a $W_n(R)$-module
since for $x\in W(N)$ and $a\in W(R)$
we have $^{V^n}a\cdot x={}^{V^n}(aF^n(x))=0$.
Thus we can define 
\[
\hat P_N=\hat W(N)\otimes_{W_n(R)}P,\qquad
\hat Q_N= ~^V\hat W(N)\otimes_{W_n(R)}T\oplus\hat W(N)\otimes_{W_n(R)}L.
\]

\begin{prop}
There is an exact sequence of abelian groups
\[
0\longrightarrow \hat Q_N\xrightarrow{\dot F-1}\hat P_N
\longrightarrow FG_n(\mathcal{P})(N)\longrightarrow 0,
\]
which defines $FG_n(\mathcal{P})(N)$; the assertion is that the first map
is injective.
The functor $N\mapsto FG_n(\mathcal{P})(N)$ on the category $\Nil_R^{(n)}$
is a truncated formal group of level $n$. 
This defines an additive and exact functor
\[
FG_n:\DDD_n(R)\to\FFF\GGG_n(R)
\]
for each ring $R$ with $pR=0$.
\end{prop}

\begin{proof}
Let $\mathcal{P}'=(P',Q',F,\dot F)$ be a display over $R$ with truncation  
$\mathcal{P}$ and let
$G=BT(\mathcal{P}')$ be the associated formal Lie group. 
By definition, for each $N\in\Nil_R$ we have an exact sequence
$$0\to \hat Q'_N\xrightarrow{\dot F-1}\hat P'_N\to G(N)\to 0.$$
If $N$ lies in $\Nil_R^{(n)}$, this sequence can be identified with
the sequence of the proposition. Thus that sequence is left exact, and
$FG_n(\mathcal{P})=G[F^n]$ is a truncated formal group of level $n$.
\end{proof}

\begin{lemma}
\label{Le-FG-Phi}
Let $G$ be a truncated $p$-divisible group of level $n$ over a ring $R$ such that
$pR = 0$.  There is a natural isomorphism
\begin{equation}\label{D2GII2e}
\text{FG}_n (\Phi_n(G)) \xrightarrow\sim G[F^{n}].
\end{equation} 
\end{lemma}

\begin{proof}
Cf.\ Remark \ref{Rk-dual}.
Assume that $G$ is the kernel of an isogeny of $p$-divisible groups,
$0\to G\to X_0\to X_1\to 0$. We obtain an exact sequence
$$
0\to G\to X_0(n)\to X_1(n)\to G\to 0.
$$
Since the functors $\Phi_n$ and $FG_n$ preserve short
exact sequences (Corollary \ref{Co-Phi-ex}) and since  
$\Phi_n(X_i(n))=\Phi_R(X_i)(n)$, we obtain an exact sequence
of finite group schemes
$$
0\to FG_n(\Phi_n(G))\to BT(\Phi(X_0))[F^n]\to BT(\Phi(X_1))[F^n]
$$
By \cite[Theorem 8.3]{Lau-Smoothness} for each $p$-divisible group $X$ over $R$
there is a natural isomorphism
\begin{equation}
\label{Eq-BT-Phi}
BT(\Phi(X))\cong \hat X.
\end{equation}
 This gives an isomorphism \eqref{D2GII2e}.
The isomorphism does not depend on the chosen resolution $X_0\to X_1$ of $G$.
Since such resolutions exist Zariski locally, the lemma follows.
\end{proof}

For a truncated display $\PPP$ of level $n$ over $R$ and a natural number $m$ we define
a finite group scheme over $R$:
\begin{equation}\label{D2GII5e}
BT_m(\mathcal{P})=FG_n(\mathcal{P})[p^m]
\end{equation}

\begin{lemma}\label{Le:BTm}
Let $R$ a ring with $pR = 0$. Let $\mathcal{P}$ be 
a truncated display of level $n$ over $R$ such that the order of nilpotence of 
$\mathcal{P}$ is $\le e$. Let $m$ be a positive integer such that $n\ge m(e+1)$. 
Then the group scheme $BT_m(\mathcal{P})$ is a truncated $p$-divisible
group of  level $m$. 
\end{lemma}

\begin{proof}
Let $\mathcal{P}'$ be a display over $R$ with truncation $\mathcal{P}$ and let
$G' = BT(\mathcal{P}')$ be the associated $p$-divisible formal group.
Lemma \ref{Le:nilpotence} implies that $G'[p]\subseteq G'[F^{e+1}]$
and thus $G'[p^m]\subseteq G'[F^{m(e+1)}] \subseteq G'[F^{n}]=FG_n(\mathcal P)$.
It follows that $BT_m(\PPP)=G'[p^m]$, which is a truncated $p$-divisible group.
\end{proof}

\begin{prop}\label{Pr-BT-Phi}
Let $R$ be a ring with $pR = 0$.
Let $G$ be a truncated $p$-divisible group of level $n$ such that the order of 
nilpotence of $G$ is $\leq e$ (see Definition \ref{Def:nilpotence-pdiv}). 
Let $m$ be a natural number such that $n \geq m(e+1)$. 
Then there is a natural isomorphism
\[
BT_m(\Phi_n(G))\cong G(m).
\]
If $\mathcal{P}$ is a truncated display of level $n$ and order of
nilpotence $\le e$ we have a canonical isomorphism
\begin{displaymath}
\Phi_m(BT_m(\mathcal{P})) \cong \mathcal{P}(m)
\end{displaymath}
\end{prop}

We note that $\Phi_n(G)$ is nilpotent of order $\le e$ by Lemma \ref{Le:nilpotence-Phi},
and therefore $BT_m(\Phi_m(G))$ is a truncated $p$-divisible group by Lemma \ref{Le:BTm}.

\begin{proof}
Since $G$ is nilpotent of order $\le e$ we have $G(1) \subseteq G[F^{e+1}]$
and thus $G(m)\subseteq G[F^{m(e+1)}]\subseteq G[F^n]$.
By taking the kernel of multiplication by 
$p^{m}$ on both sides of \eqref{D2GII2e} we obtain the first isomorphism of the proposition:
\[
\label{Eq:BT-Phi}
BT_m(\Phi_n(G)) \cong G[F^{n}][p^{m}]=G[p^m].  
\]
The second isomorphism follows using Lemma \ref{Le:Gamma} below.
\end{proof}

\begin{remark}
In the proof of Lemma \ref{Le-FG-Phi} we have used the natural isomorphism \eqref{Eq-BT-Phi}
for arbitrary $p$-divisible groups. The proof of this fact in \cite{Lau-Smoothness}
is complicated because it is difficult to relate directly the functors $\Phi$ and $BT$.

If we want to prove Lemma \ref{Le-FG-Phi} only for infinitesimal truncated $p$-divisible groups,
which is sufficient for Proposition \ref{Pr-BT-Phi}, we can modify the proof as follows.

a) One can work with resolutions $0\to G\to X_0\to X_1\to 0$ by formal
$p$-divisible groups. Such resolutions exist at least f.p.q.c.\ locally, because
f.p.q.c.\ locally $G$ extends to a formal $p$-divisible group.
By f.p.q.c.\ descent of relative truncated displays this is sufficient to construct $\Phi_n$. 
In this way we use \eqref{Eq-BT-Phi}
only for formal $p$-divisible groups, which is easier than the general case;
the proof uses the crystalline comparison of \cite{Zink-DFG} and the
equivalence, denoted by ($*$) in the following, between formal $p$-divisible groups and nilpotent displays
over arbitrary rings $R$ in which $p$ is nilpotent (\cite{Lau-Display, Lau-Smoothness}).

b) In addition one can restrict the relevant base rings $R$ and the
f.p.q.c.\ coverings $R\to R'$ such that $G_{R'}$ extends to a $p$-divisible group.
Namely, w.l.o.g.\ $R$ is an $\FF_p$-algebra of finite type, and we can
take $R'=\prod\hat R_\Fm$ where $\Fm$ runs through the maximal 
ideals of $R$.\footnote{One can also use that every truncated $p$-divisible group 
extends to a $p$-divisible group \'etale locally, but this is more difficult to show.}
Over these rings the equivalence ($*$) is already proved in \cite{Zink-DFG},
which is sufficient to deduce \eqref{Eq-BT-Phi} in the cases necessary for
the proof of Lemma \ref{Le-FG-Phi}.

c) One can also consider the following variant $\Phi_n'$ of the 
functor $\Phi_n$ restricted to infinitesimal groups:
Let $G$ be an infinitesimal truncated $p$-divisible group of level $n$.
If there is a resolution $0\to G\to X_0\to X_1\to 0$ by formal $p$-divisible groups,
let $\PPP_i$ be the nilpotent display associated to $X_i$ by the 
equivalence ($*$), and define $\Phi_n'(G)$ as the kernel of the map of
truncations $\PPP_0[n]\to\PPP_1[n]$. In general use f.p.q.c.\ descent
to define $\Phi_n'(G)$. 
Then the proof of Lemma \ref{Le-FG-Phi}
shows that $BT_m(\Phi_n'(G))\cong G(m)$ as before.
As explained in b), with appropriate modifications the proof
uses only the equivalence ($*$) in the cases covered by \cite{Zink-DFG}.
\end{remark}

Finally we extend the last two Propositions to rings where $p$ is nilpotent.

\begin{prop}
\label{Pr:BT-Phi2}
Let $S$ be a ring with $p^{m+1}S=0$ for some $m\ge 0$.
For integers $s,t,e\ge 0$ such that $t\ge (s+m)(e+1)$
and $t\ge (m(e+2)+2)(e+1)$ there is a functor
\[
 BT_{s}:  \mathcal{D}^{(e)}_{t}(S) \rightarrow \BBB\TTT^{(e)}_s(S).
\]
If $p^nW_t(S)=0$, then the composition
$\BBB\TTT_n^{(e)}(S)\xrightarrow{\Phi_t\,}\DDD_t^{(e)}(S)\xrightarrow{BT_s\,}\BBB\TTT_s^{(e)}(S)$
is isomorphic to the truncation functor.
\end{prop}

\begin{proof}
Let $R=S/pS$. By enlarging $s$ we may assume that $s\ge m(e+1)+2$.
For each truncated display $\PPP\in\mathcal\DDD_t^{(e)}(S)$ we chose a
display $\tilde\PPP$ over $S$ which lifts $\PPP$, and we set $X=BT(\tilde\PPP)$, 
or equivalently $\tilde\PPP=\Phi(X)$. We want to define $BT_s(\tilde\PPP)=X(s)$.
We have the following commutative diagram of functors, where the solid arrows
exist over $S$ and over $R$ (see Proposition \ref{AbLF1p}), 
while $BT_{s+m}$ exists only over $R$ (see Lemma \ref{Le:BTm}).
\[
\xymatrix@M+0.2em{
\BBB\TTT^{(e)} \ar[rr] \ar[d]_-\Phi &&
\BBB\TTT_{s+m}^{(e)} \ar[d]^-{\Phi_s} \\
\DDD^{(e)} \ar[r] & \DDD_t^{(e)} \ar[r] \ar@{-->}[ur]^{BT_{s+m}} & \DDD_s^{(e)}
}
\]
The diagram over $R$ commutes by Proposition \ref{Pr-BT-Phi}, 
in particular there is a canonical isomorphism $X(s+m)_R\cong BT_{s+m}(\PPP_R)$.
Corollary \ref{Co:DefXXt} and its proof imply that the following maps of sets of isomorphism
classes of deformations from $R$ to $S$ are bijective:
\[
\Def_{S/R}(\PPP_R)\to\Def_{S/R}(\PPP(s)_R)\xleftarrow{\Phi_s}\Def_{S/R}(X(s+m)_R)
\]
Assume that $\PPP'\in\mathcal\DDD_t^{(e)}(S)$ is another truncated display and
$\alpha:\PPP\to\PPP'$ is an isomorphism. We chose a lift $\tilde\PPP'$ and $X'=BT(\tilde\PPP')$
as above. Then $\beta_0=BT_{s+m}(\alpha_R)$ is an isomorphism $X(s+m)_R\cong X'(s+m)_R$.
Since $\alpha$ is a lift of $\alpha_R$ it follows that there is a lift $\beta:X(s+m)\to X'(s+m)$
of $\beta_0$. By Lemma \ref{Le:std} below, the reduction $\beta(s):X(s)\to X'(s)$ is
independent of the choice of $\beta$. Thus we have defined the functor $BT_s$ over
$S$ on the level of groupoids, and this functor preserves finite direct sums since the
construction is independent of the choice of the lift $\tilde\PPP$. 
Then a standard argument gives the full functor $BT_s$
(compare \cite{Zink-DFG}, proof of Theorem 46).

To prove the last assertion we note that for $\PPP=\Phi_t(G)$,
at least f.p.q.c\ locally we can lift $G$ to a $p$-divisible group $X$.
Then we choose $\PPP=\Phi(X)$ in the above construction, thus
$BT_s(\Phi_t(G))=G(s)$.
For a given isomorphism $\gamma:G\to G'$ with induced $\alpha:\PPP\to\PPP'$
we can choose $\beta=\alpha(s+m)$, which implies that $\beta(s)=\alpha(s)$ as required.
\end{proof}

\begin{lemma}
\label{Le:std}
Let $G$ and $H$ by truncated $p$-divisible groups over $S$ of level $u>m$.
If a homomorphism $\phi:G\to H$ reduces to zero over $R$, then the truncation
$\phi(u-m):G(u-m)\to H(u-m)$ is zero.
\end{lemma}

\begin{proof}
For every commutative affine group scheme $X$ over $S$, the kernel of
$X(S)\to X(R)$ is annihilated by $p^m$; see \cite[Lemma 3.4]{Lau-Relation}. 
If we apply this to the base change of $H$ to $S$-algebras,
it follows that $p^m\phi=0$. Therefore $\phi(u-m)=0$.
\end{proof}

\section{Vanishing homomorphisms}

Let $G$ and $G'$ be truncated $p$-divisible groups of level $n$ over a ring $R$ with $pR=0$,
with associated truncated displays $\PPP$ and $\PPP'$.
We consider the commutative group scheme of vanishing homomorphisms
\[
\uHom^o(G,G')=\Ker[\uHom(G,G')\to\uHom(\PPP,\PPP')]
\]
and the group scheme of vanishing automorphisms 
\[
\uAut^o(G)=\Ker[\uAut(G)\to\uAut(\PPP)].
\]
By \cite[Remark 4.8]{Lau-Smoothness} this is a commutative finite locally free group
scheme of rank $p^{ncd}$ where $d=\dim(\Lie G)$ and $c=\dim(\Lie G^\vee)$.

\begin{lemma}
\label{Le-uHom}
The group scheme $\uHom^o(G,G')$ is infinitesimal finite locally free of rank $p^{ncd'}$,
where $c=\dim(\Lie G^\vee)$ and $d'=\dim(\Lie G')$.
\end{lemma}

\begin{proof}
Clearly $\uHom^o(G,G')$ is an affine group scheme over $R$.
It is infinitesimal since the functor $\Phi_n$ is an equivalence 
over perfect fields. We claim that the map
\[
\uAut^o(G)\to\uEnd^o(G),\quad u\mapsto u-1
\]
is an isomorphism of schemes. Indeed, let $f\in \uEnd^o(G)(A)$ for some
$R$-algebra $A$. 
There is a nilpotent ideal $I\subset A$ such that $f+1\equiv 1$ modulo $I$.
It follows easily that $f+1$ is an automorphism.
Thus $\uEnd^o(G)$ is finite locally free. Using the decomposition of pointed $R$-schemes
\[
\uEnd^o(G\oplus G')=\uEnd^o(G)\times \uHom^o(G,G')
\times \uHom^o(G',G)\times \uEnd^o(G')
\]
it follows that $\uHom^0(G,G')$ is finite locally free. 
Its rank is locally constant on the stack $\BBB\TTT_n\times\BBB\TTT_n$.
Since the generic points of $\BBB\TTT_n$ are ordinary, to compute the rank
we may assume that each of $G$ and $G'$ is either $\ZZ/p^n\ZZ$ or
$\mu_{p^n}$. In those cases the rank is computed in \cite[Remark 4.8]{Lau-Smoothness} as desired.
\end{proof}

\begin{prop}
\label{Pr:vHom}
Under the natural isomorphism $\uHom(\ZZ/p^n,G)\cong G$ we have
$\uHom^o(\ZZ/p^n,G)\cong G[F^n]$.
\end{prop}

\begin{proof}
Both $\uHom^o(\ZZ/p^n,G)$ and $G[F^n]$ are closed subgroup schemes of $G$
which are finite locally free of rank $p^{nd}$. They coincide if $G$ is ordinary. 
Since the generic points of
the stack $\BBB\TTT_n$ are ordinary, they coincide for all $G$.
\end{proof}

We recall that $G$ is called nilpotent of order $\le e$ if $F^{e+1}$ is zero on $G(1)$;
see Definition \ref{Def:nilpotence-pdiv}.

\begin{cor}
\label{Co:reduction}
If either $G$ or $(G')^\vee$ is nilpotent of order $\le e$,
then for $n\ge m(e+1)$ the reduction map
$$
\uHom^o(G,G')\to\uHom(G(m),G'(m))
$$
is zero.
\end{cor}

\begin{proof}
We may assume that $n=m(e+1)$.
By duality it suffices to consider the case where $(G')^\vee$ is nilpotent of order $\le e$,
which means that $V^{e+1}$ is zero on $G'(1)$ and thus $V^{m(e+1)}$ is zero on $G'(m)$.
It follows that the subgroup scheme
\[
G'[F^{m(e+1)}]=\Image(V^{m(e+1)})
\]
of $G'$ maps to zero under the surjection $p^{me}:G'\to G'(m)$.
Thus for $G=\ZZ/p^n$ the corollary follows from Proposition \ref{Pr:vHom}.

If $G$ is arbitrary we consider an element $u\in\uHom^o(G,G')(A)$ for some $R$-algebra $A$. 
Let $A\to B$ be a ring homomorphism and let $a\in G(B)$. The composition
$
\ZZ/p^n\xrightarrow{\; a\;} G_B \xrightarrow{\; u\;} (G')_B
$
lies in $\uHom^o(\ZZ/p^n,G')(B)$. By the first case, its reduction to level $m$ is trivial as desired.
\end{proof}

Let $\BBB\TTT_n$ and $\DDD_n$ be the algebraic stacks over $\FF_p$ 
of truncated $p$-divisible groups resp.\ truncated displays of level $n$.
Let $\BBB\TTT_n^{(e)}\subset\BBB\TTT_n$ and $\DDD_n^{(e)}\subset\DDD_n$ 
be the closed substacks where the nilpotence order is $\le e$. 
For $n\ge m$ we have the following commutative diagram, where $\tau$ are the truncation maps.
\begin{equation}
\label{Eq-stack-square}
\xymatrix@M+0.2em@C+1em{
\BBB\TTT^{(e)}_n \ar[r]^-{\Phi_n} \ar[d]_-\tau & \DDD_n^{(e)} \ar[d]^-\tau \\
\BBB\TTT^{(e)}_m \ar[r]^-{\Phi_m} & \DDD_m^{(e)}}
\end{equation}
Here all morphisms are smooth and surjective because the diagram is the inverse
image under the inclusion $\DDD_m^{(e)}\to \DDD_m$
of the corresponding diagram without $^{(e)}$, in which all morphisms are smooth.

\begin{lemma}
\label{Le:Gamma}
For given $n\ge m$ and $e\ge 0$, there is a morphism of stacks
$\Gamma:\DDD_n^{(e)}\to\BBB\TTT_m^{(e)}$ such that
$\Gamma\circ\Phi_n\cong\tau$ iff the reduction map
$$\rho:\uAut^o(G)\to\uAut(G(m))$$
is zero for all $G$ which are nilpotent of order $\le e$.
In that case $\Gamma$ is unique up to unique isomorphism,
and we also have $\Phi_m\circ\Gamma\cong\tau$.
\end{lemma}

\begin{proof}
We consider truncated $p$-divisible groups of fixed height $h$
and truncated displays of rank $h$ without changing the notation.
Let $X=\Spec R\to\BBB\TTT_n^{(e)}$ be a smooth presentation given
by a truncated $p$-divisible group $G$ of level $n$ over $R$. 
Since all arrows in \eqref{Eq-stack-square} are smooth and surjective,
we also get smooth presentations of the other three stacks.
Let $\PPP=\Phi_n(G)$ be the truncated display associated to $G$.
Let $G_1,G_2$ and $\PPP_1,\PPP_2$ over $X\times X$ be the inverse
images of $G$ and $\PPP$ under the two projections.
We get a commutative diagram of groupoids over $X$:
\begin{equation}
\label{Eq-groupoid-square}
\xymatrix@M+0.2em@C+1em{
\uIsom(G_1,G_2) \ar[d]_\tau \ar[r]^-{\Phi_n} &
\uIsom(\PPP_1,\PPP_2) \ar[d]^\tau \\
\uIsom(G_1(m),G_2(m)) \ar[r]^-{\Phi_m} &
\uIsom(\PPP_1(m),\PPP_2(m))
}
\end{equation}
The associated diagram of stacks is \eqref{Eq-stack-square}.
The morphism $\Phi_n$ in \eqref{Eq-groupoid-square} is a
torsor under $\uAut^o(G_1)$ by \cite[Theorem 4.7]{Lau-Smoothness}.
Thus there is a morphism of schemes $\Gamma:\uIsom(\PPP_1,\PPP_2)\to\uIsom(G_1(m),G_2(m))$
such that the upper triangle commutes iff $\rho_1:\uAut^o(G_1)\to\uAut(G_1(m))$ is trivial.
In that case $\Gamma$ is unique, and the lower triangle commutes as well.
Since $G_1$ is the inverse image of the universal group under a faithfully flat map
$X\times X\to\BBB\TTT_n^{(e)}$, the reduction map $\rho_1$ is trivial iff $\rho$ is
zero for all $G$ which are nilpotent of order $\le e$.
This proves the lemma.
\end{proof}

\begin{prop}
\label{Pr:Gamma}
For $n\ge m(e+1)$ there is a functor $BT_m:\DDD_n^{(e)}\to\BBB\TTT_n^{(e)}$
such that $BT_m\circ\Phi_n$ is isomorphic to the truncation functor. This functor
$BT_m$ is unique up to unique isomorphism, and $\Phi_m\circ BT_m$ is isomorphic
to the truncation functor as well.
\end{prop}

\begin{proof}
Use Corollary \ref{Co:reduction} for $\uEnd(G)$ and Lemma \ref{Le:Gamma}.
\end{proof}

\begin{remark}
By a standard argument (see \cite{Zink-DFG}, proof of Theorem 46),
the functor $BT_m$ of groupoids extends to a functor of additive categories,
which necessarily coincides with the functor of Proposition \ref{Pr-BT-Phi}.
\end{remark}

It is easy to see that the functors $BT_m$ for varying $m$
are compatible in such a way that they form an inverse to the functor
$$
\Phi:\varprojlim_n\BBB\TTT_n^{(e)}\to\varprojlim_n\DDD_n^{(e)},
$$
so this is an equivalence. This gives a new proof of the equivalence
between formal $p$-divisible groups and nilpotent displays over rings with $pR=0$.
The general case of rings in which $p$ is nilpotent follows easily by
deformation theory.


\appendix
\section{Descent for truncated displays} 

\begin{prop}
Let  $R \rightarrow S$ be a faithfully flat homomorphism
of rings in which $p$ is nilpotent. Then
the Cech complex 
\begin{displaymath}
   \mathcal{W}_n(R) \rightarrow \mathcal{W}_n(S) 
\begin{array}{c}
\rightarrow \\[-3mm]
\rightarrow 
\end{array}
\mathcal{W}_n(S\otimes_RS) 
\begin{array}{c}
\rightarrow \\[-3mm]
\rightarrow \\[-3mm]
\rightarrow  
\end{array}
\mathcal{W}_n(S\otimes_RS\otimes_RS) \ldots 
\end{displaymath}
is acyclic.
\end{prop}
{\bf Proof:} To the simplicial complex above we have also the associated
chain complex which will be denoted by $\mathcal{CW}_n(S/R)$.

Let $R[p]$ be the kernel of the multiplication by $p$. By
tensoring with $\otimes_{R}S$:
\begin{displaymath}
   S[p] = R[p] \otimes_{R} S.
\end{displaymath}
By descent theory we know that the Cech complex of the $R$-module
$R[p]$ relative to the covering $\Spec S \rightarrow \Spec R$ is
acyclic: 
\begin{displaymath}
   R[p] \rightarrow S[p]
\begin{array}{c}
\rightarrow \\[-3mm]
\rightarrow 
\end{array}
(S\otimes_RS)[p] 
\begin{array}{c}
\rightarrow \\[-3mm]
\rightarrow \\[-3mm]
\rightarrow  
\end{array}
(S\otimes_RS\otimes_RS)[p]  \ldots . 
\end{displaymath}
Let $\mathcal{C}(S/R)[p]$ be the associated chain complex.  Using the
remarks after \eqref{Ab1e} we obtain an exact sequence of complexes
\begin{displaymath}
 0 \rightarrow \mathcal{C}(S/R)[p] \rightarrow \mathcal{C}W_{n+1}(S/R)
 \rightarrow  \mathcal{CW}_n(S/R) \rightarrow 0.
\end{displaymath}
The definition and exactness of the complex in the middle follows from
\cite{Zink-DFG} Lemma 30. This concludes the proof of the Proposition.

The notion of a $W$-descent datum \cite{Zink-DFG} applies to $\mathcal{W}_n(R)
\rightarrow \mathcal{W}_n(S)$ and is then called a
$\mathcal{W}_n$-descent datum.  

\begin{prop}
\label{Pr:descent-module}
With the assumptions of the last Proposition let $P$ be a finitely
generated projective  $\mathcal{W}_n(S)$-module with a descent datum
\begin{equation}\label{AbstiegDsp0e}
  \nabla:  \mathcal{W}_n(S \otimes_R S) _{p_1, \mathcal{W}_n(S)} P \rightarrow
   \mathcal{W}_n(S \otimes_R S) _{p_2, \mathcal{W}_n(S)} P.  
\end{equation}
The associated chain complex $\mathcal{C}_{\mathcal{W}_n}(P;S/R)$
(compare: \cite{Zink-DFG} (43)) 
\begin{displaymath}
   P \rightarrow \mathcal{W}_n(S \otimes_R S)
   \otimes_{\mathcal{W}_n(S)} P \rightarrow \mathcal{W}_n(S \otimes_R
   S \otimes_R S) \otimes_{\mathcal{W}_n(S)} P \rightarrow \ldots
\end{displaymath}
is exact. Here the $\mathcal{W}_n(S)$-module structure on
$\mathcal{W}_n(S \otimes_R \ldots \otimes_R S)$ is via the last factor
of the tensorproduct.  

If $P_0$ is the kernel of the first arrow we have a
canonical isomorphism 
\begin{displaymath}
   \mathcal{W}_n(S) \otimes_{\mathcal{W}_n(R)} P_0 \rightarrow P.
\end{displaymath}
\end{prop}
{\bf Proof:} We begin with a general remark. Let $R'$ be an
$R$-algebra. We denote by $\mathfrak{n}_{R'} \subset R'$ the ideal of
all elements annihilated by $p$. We set $\mathfrak{n} =
\mathfrak{n}_{R}$. If $R \rightarrow R'$ is flat then
$\mathfrak{n}_{R'} = \mathfrak{n} \otimes_{R} R' = \mathfrak{n}R'$. 

We denote by $\mathbf{w}_n : \mathcal{W}_n(R') \rightarrow
R'/\mathfrak{n}R'$ the homomorphism induced by the Witt polynomial of
degree $p^n$. For a $\mathcal{W}_n(R)$-module $P_0$ we set $\bar{P}_0
= R/\mathfrak{n} \otimes_{\mathbf{w}_n, \mathcal{W}_n(R)} P_0$. We
have the isomorphism
\begin{displaymath}
   R'/\mathfrak{n}R' \otimes_{\mathbf{w}_n, \mathcal{W}_n(R')}
   (\mathcal{W}_n(R') \otimes_{\mathcal{W}_n(R)} P_0) \cong
   R'/\mathfrak{n}R' \otimes_{R/\mathfrak{n}} \bar{P}_0
\end{displaymath}
If we tensor the descent datum \eqref{AbstiegDsp0e} with $W_n(S
\otimes_{R} S) \otimes_{\mathcal{W}_n(S \otimes_{R} S)}$ we obtain a
$W_n$-descent datum on $W_n(S) \otimes_{\mathcal{W}_n(R)} P$ and if we
tensor with $(S/\mathfrak{n}S \otimes_{R/\mathfrak{n}} S/\mathfrak{n}S) 
\otimes_{\mathbf{w}_n, \mathcal{W}_n(S \otimes_{R} S)}$ we obtain a
descent datum on the $S/\mathfrak{n}S$-module $\bar{P} =
S/\mathfrak{n}S \otimes_{\mathbf{w}_n, \mathcal{W}_n(S)} P$. 

By the definition of $\mathcal{W}_n(S)$ we have an exact sequence of
$\mathcal{W}_n(S)$-modules
\begin{displaymath}
   0 \rightarrow (S/\mathfrak{n}S)_{[\mathbf{w}_n]}
   \overset{V^n}{\rightarrow } 
   \mathcal{W}_n(S) \rightarrow W_n(R) \rightarrow 0.
\end{displaymath}
By tensoring with $P$ we obtain the exact sequence
\begin{equation}\label{Abstieg2e}
   0 \rightarrow \bar{P} \overset{V^n}{\rightarrow } P
\rightarrow W_n(S) \otimes_{\mathcal{W}_n(S)} P \rightarrow 0.
\end{equation}
We have a commutative diagram

\begin{displaymath}
   \begin{CD}
(S/\mathfrak{n}S \otimes_{R/\mathfrak{n}} S/\mathfrak{n}S) 
\otimes_{\mathbf{w}_n \circ p_1, \mathcal{W}_n(S)} P @>{\id \otimes \nabla}>>
(S/\mathfrak{n}S \otimes_{R/\mathfrak{n}} S/\mathfrak{n}S) 
\otimes_{\mathbf{w}_n \circ p_2, \mathcal{W}_n(S)} P\\
@V{V^n}VV  @VV{V^n}V\\
\mathcal{W}_n(S \otimes_{R}S) \otimes_{p_1, \mathcal{W}_n(S)} P @>>{\nabla}>
\mathcal{W}_n(S \otimes_{R}S) \otimes_{p_2, \mathcal{W}_n(S)} P
   \end{CD}
\end{displaymath}
Therefore the exact sequence \eqref{Abstieg2e} is compatible with the
descent data and yields an exact sequence of complexes:
\begin{displaymath}
   0 \rightarrow \mathcal{C}(\bar{P};
   (S/\mathfrak{n}S)/(R/\mathfrak{n})) \rightarrow 
   \mathcal{C}_{\mathcal{W}_n}(P;S/R) \rightarrow
   \mathcal{C}_{W_n}(P;S/R) \rightarrow 0
\end{displaymath} 
The first complex is the complex associated to the descent datum on
the $S/\mathfrak{n}S$-module $\bar{P}$ relative to $R/\mathfrak{n}
\rightarrow S/\mathfrak{n}S$. By \cite{Zink-DFG} and usual descent we
know that 
execpt for the complex in the middle we have $H^i = 0$ for $i \geq
1$. Then this holds also for the complex in the middle. Taking $H^0$
we obtain the exact cohomology sequence
\begin{displaymath}
0 \rightarrow  \bar{P}_0 \rightarrow P_0 \rightarrow \breve{P}_0
\rightarrow 0.
\end{displaymath}
By $W_n$-descent we know that the natural map
\begin{displaymath}
   W_n(S) \otimes_{W_n(R)} \breve{P}_0 \rightarrow W_n(S)
   \otimes_{\mathcal{W}_n(S)} P 
\end{displaymath}
is an isomorphism. Let $E$ be a finitely generated projective
$\mathcal{W}_n(R)$-module which lifts the $W_n(R)$-module
$\breve{P}_0$. We find a factorization $E \rightarrow P_0 \rightarrow
\breve{P}_0$. By the Lemma of Nakayama we conclude that
\begin{displaymath}
   \mathcal{W}_n(S) \otimes_{\mathcal{W}_n(R)} E \rightarrow P
\end{displaymath}
is an isomorphism. Since the last arrow is compatible with the descent
data on both sides we obtain an isomorphism of complexes
\begin{displaymath}
   \mathcal{C}_{\mathcal{W}_n}(\mathcal{W}_n(S)
   \otimes_{\mathcal{W}_n(R)} E; S/R) \rightarrow
   \mathcal{C}_{\mathcal{W}_n}(P;S/R).  
\end{displaymath}
It follows that $E\to P_0$ is an isomorphism.
This proves the Proposition. 

\begin{cor}
\label{Cor:descent-module}
For $R\to S$ as above let $P_1$ be a finitely presented $\WWW_n(R)$-module such that
$P=\WWW_n(S)\otimes_{\WWW_n(R)}P_1$ is projective.
Then $P_1$ is projective.
\end{cor}

\begin{proof}
The module $P$ carries a natural descent datum. 
Proposition \ref{Pr:descent-module} gives a finitely generated
projective $\WWW_n(R)$-module $P_0$, and the natural map $P_1\to P$
factors over a homomorphism $g:P_1\to P_0$ that
becomes bijective over $\WWW_n(S)$. By Nakayama's lemma $g$ is surjective.
Then $P_1\cong P_0\oplus N$ where $N$ is finitely generated and 
$\WWW_n(S)\otimes_{\WWW_n(R)}N=0$. By Nakayama's lemma it follows that $N=0$.
\end{proof}

Let $R \rightarrow S$ be a faithfully flat ring homomorphism as
before. Let $\mathcal{P} =  (P,Q,\iota, \epsilon, F, \dot{F})$ be a
truncated display of level over $R$. We denote the base change to $S$
by $\mathcal{P}_S = (P_S, Q_S, \iota_S, \epsilon_S, F_S,
\dot{F}_S)$. There is a natural homomorphism $\mathcal{P} \rightarrow
\mathcal{P}_S$ which is obvious in terms of a normal decomposition. 
We obtain a simplicial complex
\begin{equation}\label{AbstiegDsp1e}
   \mathcal{P} \rightarrow \mathcal{P}_S 
\begin{array}{c}
\rightarrow \\[-3mm]
\rightarrow 
\end{array}
\mathcal{P}_{S \otimes_R S} 
\begin{array}{c}
\rightarrow \\[-3mm]
\rightarrow \\[-3mm]
\rightarrow  
\end{array}
\mathcal{P}_{S \otimes_R S \otimes_R S} \ldots.
\end{equation}
\begin{prop}
The simplicial complex \eqref{AbstiegDsp1e} induces exact chain
complexes

\begin{displaymath}
\begin{array}{l}
   0 \rightarrow P \rightarrow P_S \rightarrow P_{S \otimes_R S}
   \rightarrow P_{S \otimes_R S \otimes_R S} \ldots \\[3mm]
   0 \rightarrow Q \rightarrow Q_S \rightarrow Q_{S \otimes_R S}
   \rightarrow Q_{S \otimes_R S \otimes_R S} \ldots .
\end{array} 
\end{displaymath}
\end{prop}
{\bf Proof:} We know that 
$P_S = \mathcal{W}_n(S) \otimes_{\mathcal{W}_n(R)} P$. Therefore we
obtain the first exact sequence from the first Proposition. To obtain
the second exact sequence we choose a normal decomposition $P = T
\oplus L$. Then $Q_S = I_{n+1}(S) \otimes_{\mathcal{W}_n(R)} T \oplus
\mathcal{W}_n(S) \otimes_{\mathcal{W}_n(R)} L$. The exactness of the
second sequence amounts therefore to that of
\begin{displaymath}
   I_{n+1}(R) \otimes_{\mathcal{W}_n(R)} T \rightarrow I_{n+1}(S)
   \otimes_{\mathcal{W}_n(R)} T \rightarrow 
I_{n+1}(S \otimes_R S) \otimes_{\mathcal{W}_n(R)} T \rightarrow 
    \ldots. 
\end{displaymath}
But this is clear. 

We have the notion of descent datum relative to $S/R$ for a truncated 
display $\tilde{\mathcal{P}}$ over $S$. The Proposition shows that the
functor which associates to a truncated display over $R$ the base
change to a truncated display over $S$ with a descent datum is fully
faithful. 

\begin{prop}
\label{Pr:descent-disp}
Let $R \rightarrow S$ be faithfully flat. The base change functor
induces an equivalence of the category $\mathcal{D}_n(R)$ of truncated
displays of level $n$ over $R$ with the category of truncated
displays of level $n$ over $S$ endowed with a $\mathcal{W}_n$-descent datum
relative to $S/R$ (see: \eqref{AbstiegDsp0e}).
\end{prop}
{\bf Proof:} It follows from the last Proposition that this functor is
fully faithful. Therefore it suffices to show that a
$\mathcal{W}_n$-descent datum is always effective.

Following \cite{Lau-Smoothness} we begin to prove a related descent
result. We call a 
$(\mathcal{W}_n(R), I_{n+1})$-module $(P,Q,\iota, \epsilon)$ which
admits a normal decomposition a truncated pair. In particular the
$R$-modules $P/\iota(Q)$ and $Q/\Image \epsilon$ are 
projective finitely generated $R$-modules. 

The first lines of the proof of Proposition  \ref{Pr-norm-dec} show
that a second truncated pair $(P',Q',\iota', \epsilon')$ is isomorphic
to $(P,Q,\iota, \epsilon)$ iff there are isomorphisms of $R$-modules 
\begin{displaymath}
   P/\iota(Q) \cong P'/\iota(Q'), \quad Q/\Image \epsilon \cong
   Q'/\Image \epsilon'.  
\end{displaymath}
More precisely, if two such isomorphisms are given, they are induced
by an isomorphism
\begin{displaymath}
   (P,Q,\iota, \epsilon) \rightarrow (P',Q',\iota', \epsilon').
\end{displaymath}

We fix projective finitely generated $R$-modules $\bar{T}$ and
$\bar{L}$. Let $\mathcal{F}$ be the cofibered category over the
category of $R$-algebras $S_1$, such that an object of
$\mathcal{F}_{S_1}$ is a truncated pair $(P,Q,\iota, \epsilon)$ over
$S_1$ endowed with isomorphisms
\begin{displaymath}
   P/\iota(Q) \cong S_1 \otimes_R \bar{T}, \quad Q/\Image \epsilon \cong
   S_1 \otimes_R \bar{L}.
\end{displaymath}
By the remark above any two objects in $\mathcal{F}_{S_1}$ are
isomorphic. 

We fix an object $(P_0,Q_0, \iota_0, \epsilon_0) \in
\mathcal{F}_R$. We denote by $\mathcal{A}_{S_1}$ the automorphism of
the base change $(P_0,Q_0,\iota_0, \epsilon_0)_{S_1}$. 
By \cite{Milne} Chapt.III \S4 the set of isomorphism classes of descent data on $(P_0,Q_0, \iota_0,
\epsilon_0)_S$ is bijective to the nonabelian Cech cohomology set $\check{H}^1(S/R,
\mathcal{A})$.  We will show below that this pointed set is trivial.
Equivalently this says that any desent datum in $\mathcal{F}$
relative to $S/R$ is effective. 

We can now prove the Proposition. Let $\mathcal{P} = (P,Q,\iota,
\epsilon, F, \dot{F})$ be a truncated display over $S$ which is
endowed with a descent datum. We denote by $\breve{P} = (P,Q,\iota,
\epsilon)$ the associated truncated pair. The descent datum induces a
descent datum on the $S$-modules $P/\iota(Q)$ and $Q/\Image
\epsilon$. We find projective finitely generated projective
$R$-modules $\bar{T}$ and $\bar{L}$ and isomorphism which are
compatible with the descent data on both sides
\begin{displaymath}
   P/\iota(Q) \cong S \otimes_R \bar{T}, \quad Q/\Image \epsilon \cong
   S \otimes_R \bar{L}. 
\end{displaymath} 
This makes $\breve{\mathcal{P}}$ an object in $\mathcal{F}_S$ and the
descent datum a morphism in $\mathcal{F}_{S\otimes_R S}$. Therefore we
know that the descent datum is effective. Because of the fully
faithfulness of descent for truncated pairs the morphisms $F$ and
$\dot{F}$ descent too. 

It remains to show the triviality $\check{H}^1(S/R, \mathcal{A})$. 
We vary now $n$ and we set $\mathcal{F}_n = \mathcal{F}$.
Assume that $n = 1$. In this case we consider the image
$\bar{\mathcal{A}_1}$ by the map $\mathcal{A}_1 \rightarrow \Aut P_0 
\otimes_{\mathcal{W}_1(R)} R$. This is just the additive group of an
$R$-module and therefore the Cech cohomology of $\bar{\mathcal{A}_1}$
is trivial.The matrix representation of an element in
the kernel of $\mathcal{A}_1 \rightarrow \bar{\mathcal{A}_1}$ has
the form $E + \mathcal{X}$ where $E$ is the unit matrix and $\mathcal{X}$ 
a matrix with coefficients in $R$-modules. 

In the case $pR = 0$ we
have $(E + \mathcal{X})(E + \mathcal{X'}) = E + \mathcal{X} +
\mathcal{X'}$. This shows the the Cech cohomology of the kernel is
trivial. By the exact cohomology sequence for Cech cohomology of
presheaves we obtain that $H^{1}(S/R,\mathcal{A}_1)$ is trivial. In the general
case we consider the filtration of $R$ by $p^mR$ and obtain the
triviality too.  

Let now $n > 1$. 
We denote by $(P'_0, Q'_0, \iota'_0, \epsilon'_0) \in
\mathcal{F}_{n-1}$ the truncation of $(P_0, Q_0, \iota_0,
\epsilon_0)$. We denote by $\mathcal{A}_{n-1}$ its automorphism
group. Let $\mathcal{K}$ be the kernel of the natural surjection of
presheaves $\mathcal{A}_n \rightarrow \mathcal{A}_{n-1}$. By induction
it suffices to show that the Cech cohomology of $\mathcal{K}$ is
trivial. Again we look at the matrix interpretation of
$\mathcal{K}$. The matrices in $\mathcal{K}$ are of the form $E +
\mathcal{X}$ where $\mathcal{X}$ has coefficients in an $R$-module. In
the case $pR = 0$ we have simply the additive group of this module and
therefore the Cech cohomology is trivial. In the general case we
consider the filtration above. By the exact cohomology sequence we
obtain the the triviality of $H^{1}(S/R, \mathcal{A}_n)$. This proves
the Proposition.

\end{document}